\documentclass{article}
\usepackage[english]{babel}
\usepackage{amssymb,amsmath,amscd,amstext, comment}
\usepackage{amsfonts}
\usepackage{amsthm}
\usepackage{xypic}
\usepackage[textsize=footnotesize]{todonotes}
\input{xy}
\xyoption{all} 

\newtheorem{theorem}{${}$\hspace{-4.5mm}Theorem}[section]
\newtheorem{corollary}[theorem]{${}$\hspace{-4.5mm} Corollary}
\newtheorem{proposition}[theorem]{${}$\hspace{-4.5mm} Proposition}
\newtheorem{definition}[theorem]{${}$\hspace{-4.5mm} Definition}
\newtheorem{lemma}[theorem]{${}$\hspace{-4.5mm} Lemma}

\newtheorem{remark}{${}$\hspace{-4.5mm}Remark}[section]

\def\wht{\,\,\widehat{}\,\,}

\title{Cocommutative coalgebras: homotopy theory and Koszul duality\footnote{This work was supported by EPSRC grants EP/J008451/1 and EP/J00877X/1 }}

\author{{JOSEPH CHUANG}\footnote{Department of Mathematics, City University London, London EC\textup{1V 0}HB}\,,\,\,  ANDREY LAZAREV\footnote{Mathematics and Statistics, Lancaster University, Lancaster, LA\textup{1 4}YF } \,\,\& W.H. MANNAN\footnote{Mathematics and Statistics, Lancaster University, Lancaster, LA\textup{1 4}YF }}

\date{\vspace{-5ex}}

\begin{document}

\maketitle

\begin{abstract}
We extend a construction of Hinich to obtain a closed model category structure on all differential graded cocommutative coalgebras over an algebraically closed field of characteristic zero.  We further  show that the Koszul duality between commutative and Lie algebras extends to a Quillen equivalence between cocommutative coalgebras and formal coproducts of curved Lie algebras.
\end{abstract}

\section{Introduction}

Differential graded (dg) coalgebras arise naturally as invariants of topological spaces, for example as the normalized singular chains of a space. They also serve as representing objects for formal deformation functors \cite{Hini, Mane, Prid} and feature prominently in rational homotopy theory \cite{Quil, Neis}. As such it is natural to ask if they may be placed in the framework of a closed model category (CMC), at least in the case of cocommutative dg coalgebras over a field of characteristic zero. The first result of this kind is
due to D. Quillen \cite{Quil} under a rather strong connectivity assumption; this assumption was subsequently removed by Hinich in \cite{Hini}. The crucial difference between these two approaches is that Quillen defined weak equivalences to be quasi-isomorphisms whereas Hinich considered a finer (i.e. harder to satisfy) notion of a \emph{filtered quasi-isomorphism}. A particularly nice feature of Hinich's model is the so-called Koszul duality: it turns out to be Quillen equivalent to the category of dg Lie algebras, again without any grading restrictions.

Hinich's construction, while being a vast generalization of Quillen's, is not completely general in that only \emph{conilpotent} dg coalgebras were allowed.
Our goal is to extend Hinich's construction as well as Koszul duality to \emph{all} cocommutative dg coalgebras, not necessarily conilpotent. We give a rather complete answer in the case when the ground field is algebraically closed.

There are by now quite a few papers devoted to the study of homotopy theory of dg coalgebras. One can  try to extend Hinich's approach to dg coalgebras over other operads, e.g. coassociative dg coalgebras. The corresponding theory was constructed in \cite{Lefe, Posi}; it was further generalized in \cite{Vall} to dg coalgebras over an arbitrary Koszul operad. In these papers coalgebras are still assumed to be conilpotent. It is not clear at present how one can extend these results to non-conilpotent dg algebras, even in the associative case. It is interesting that Positselski nevertheless \cite{Posi} succeeded in constructing a CMC structure on the category of dg comodules over an arbitrary dg coalgebra.

There have been various attempts at constructing  a CMC structure on the category of dg coalgebras taking weak equivalences to be quasi-isomorphisms. Technically, it could be viewed as a Bousfield localization of a Hinich-type CMC. In the absence of the latter one can try to construct such a CMC by transfer from the category of dg vector spaces. In this way, Getzler and Goerss constructed a CMC of non-negatively graded coassociative dg coalgebras, \cite{Getz}; a more abstract approach was taken in the recent paper \cite{Stan}. The paper \cite{Hess} takes a categorical approach allowing to prove a CMC structure on dg comodules over a coring satisfying certain conditions. The paper \cite{Aubr} constructs a CMC on the category of dg coalgebras over a quasi-free operad (however note that the operads of associative or commutative algebras do not fall into this framework). Finally, the series of papers by J. Smith \cite{Smit,Smit1,Smit2} claimed to construct a transferred CMC structure on dg coalgebras over fairly arbitrary cooperads. It follows from our results that  some additional assumptions (e.g. those made in \cite{Aubr}) are necessary since the cofree coalgebra functor fails to be exact in the simplest possible case: that of cocommutative dg coalgebras over a field of characteristic zero. Strikingly, it \emph{is} exact in the coassociative context, a phenomenon for which we lack a really satisfactory explanation, cf. Remark \ref{comass} below. It follows that there cannot be a transferred CMC on cocommutative dg coalgebras, as opposed to the coassociative dg coalgebras (even if one imposes the characteristic zero assumption).


The category of cocommutative coalgebras may be identified via dualisation over the ground field with the opposite category to the category of pseudo--compact commutative algebras, cf. \cite{Gabr, Dema} for a detailed study of pseudo-compact algebras in a non-differential context. This point of view is taken, e.g., in the papers \cite{Laza1, Laza2} and we adopt it here as well; it is quite convenient, particularly when dealing with cofree coalgebras (whose duals are suitably completed symmetric algebras). If a dg coalgebra is conilpotent, then the corresponding dual algebra is local in the differential sense, i.e. it has a unique dg ideal.
We refer to such algebras as Hinich algebras.

In \S\ref{pseudosec} we extend arguments of Demazure concerning commutative algebras over  algebraically closed fields \cite{Dema} to the graded commutative case, thus showing that any pseudo--compact  algebra is the direct product of local pseudo--compact  algebras.  However Demazure's methodology does not take account of the differential, so as well as Hinich algebras, pseudo--compact dg algebras include algebras where the maximal ideal is not closed under the differential.

In \S\ref{extsec} we extend the CMC stucture on Hinich algebras to the category of all pseudo--compact dg algebras with a unique maximal graded ideal, possibly non-differential, which we refer to as the extended Hinich category.  One may visualize this as passing from the category of pointed connected spaces to the category of pointed connected spaces together with the empty set.  The role of the empty set is played by an acyclic algebra which we denote $\Lambda(x)$.  However we are not merely adding an object analogous to the empty set, but rather objects analogous to the Cartesian  products of the empty set with a connected space.  Unlike their topological analogues, these objects are not all isomorphic.

As a result, this extension of the CMC structure on Hinich algebras is nontrivial, and depends on the particular properties of $\Lambda(x)$ in the extended Hinich category.   However the structure that we obtain is in some sense natural, being the unique extension which preserves certain intuitions, such as the inclusion of the empty set in a point being both a fibration and a cofibration.

Next in \S\ref{prodsec} we show that the category of formal products of objects in a CMC is itself a CMC.   Continuing our topological analogy, this corresponds to passing from pointed connected spaces to (multiply) pointed disconnected spaces.  In fact the CMC structure on formal products is quite natural, with the properties of a map being a cofibration or weak equivalence being tied to the corresponding properties on the components of the map in the original category.  The only surprise is that the property of a  map being a fibration, determined by the right lifting property, is not tied to the components of the map being a fibration.

In particular we conclude that the category of all pseudo-compact commutative dg algebras, being the category of formal products in the extended Hinich category, is itself a CMC.   We attempt a more abstract view of the situation by relating the CMC structure on all pseudo--compact algebras to the one on Hinich algebras via a Quillen adjunction, which may be defined independently of the factorization in  \S\ref{pseudosec} in the commutative case.

In \cite{Hini} the CMC structure was transferred via a Quillen equivalence from the category of dg Lie algebras.  The natural extension of this Quillen equivalence to our extended Hinich category is to the category of curved Lie algebras. In \S\ref{curvedsec} we describe the CMC structure on curved Lie algebras and how the Quillen equivalence extends to it. It follows that the category of all pseudo-compact commutative dg algebras is Quillen equivalent to the opposite category of formal coproducts of curved Lie algebras. One can view this (anti)-equivalence as the commutative-Lie Koszul duality extended to all pseudo--compact commutative dg algebras (or, equivalently, all cocommutative dg coalgebras).

We remark that the construction of the CMC structure on curved dg Lie algebras may be applied verbatim to yield a CMC structure on curved associative dg algebras.  This raises the possibility of transferring this CMC structure to the category of all local pseudo--compact coassociative dg coalgebras.

\section{Pseudo--compact algebras} \label{pseudosec}
Henceforth let $k$ be an algebraically closed field of characteristic 0. We work in the underlying category of differential $\mathbb Z$-graded $k$-vector spaces. Algebras and coalgebras (which will always be graded commutative) are assumed to be over $k$ as are unlabeled tensor products.   Let $\mathcal{V}$ denote the category of counital (cocommutative) dg coalgebras over $k$.

Let $V$ be an object in $\mathcal{V}$.  Given any $v \in V$, the counit and coassociativity of $V$ imply that $v$ is contained in a finite dimensional graded ideal in $V$ \cite[Lemma 1.1 \& Lemma 1.2]{Getz}.  By cocommutativity this is necessarily a finite dimensional subcoalgebra.  (In the non-commutative case such a subcoalgebra containing $v$ may still be constructed by other means).  Thus $V$ is the union of its finite dimensional graded subcoalgebras and enlarging these subcoalgebras if necessary one can assume that they are closed with respect to the differential and still finite-dimensional. It follows that the linear dual of $V$ over $k$ (the algebra henceforth denoted $V^*$) is the inverse limit of finite dimensional
(commutative) dg algebras (namely the linear duals of the finite dimensional dg subcoalgebras of $V$).

As such, the algebra $V^*$ is endowed with a topology (regarding finite dimensional algebras as discrete spaces).  Given a morphism in $\mathcal{V}$: $f\colon V_1 \to V_2$, its dual is a continuous algebra morphism $f^*\colon V_2^* \to V_1^*$.  Conversely any such continuous linear map is induced by a morphism in $\mathcal{V}$.

\begin{definition}
A \emph{pseudo--compact dg algebra} over $k$ is the inverse limit of a diagram of unital finite dimensional dg algebras  over $k$.  A morphism of pseudo--compact dg algebras is a unit preserving dg-algebra morphism which is continuous with respect to the induced topology.
\end{definition}

Then $V^*$ is a pseudo--compact dg algebra.  In fact any pseudo--compact dg algebra arises as the dual of a dg coalgebra.  To recover this coalgebra, simply take the topological dual. Thus we may identify $\mathcal{V}^{\rm op}$ with the category of pseudo--compact dg algebras.  We will show that an arbitrary pseudo--compact dg algebra may be identified (both algebraically and topologically) with the direct product of local pseudo--compact dg algebras:

\begin{definition}\label{loco}
A \emph{local pseudo--compact dg algebra} is a pseudo--compact dg algebra having a unique maximal graded ideal, possibly not closed with respect to the differential.
\end{definition}

 We consider the following two types of local pseudo--compact dg algebra:

\begin{definition}
A \emph{Hinich algebra} is the linear dual of a dg coalgebra in the sense of \cite[2.1]{Hini}.  Specifically it is a local pseudo--compact dg algebra whose maximal ideal is closed under the differential.
\end{definition}

\begin{definition}
An \emph{acyclic algebra} is a local pseudo--compact dg algebra 
in which every closed element is a boundary.
\end{definition}

In fact these are all the local pseudo--compact dg algebras:
\begin{lemma} \label{twotypes}
Any local pseudo--compact dg algebra is either a Hinich algebra or an acyclic algebra.
\end{lemma}

\begin{proof}
Let $A$ be a local pseudo--compact dg algebra.  Suppose $A$ is not a Hinich algebra.  Then the maximal ideal $M$ is not closed under the differential.  That is we have some $dm\notin M$ for some  homogeneous $m \in M$.  Then $dm$ is invertible as otherwise it would generate an ideal, necessarily contained in $M$.  Let $wdm=1$ for some homogeneous $w\in A$.  We have:$$
0=d1=d(w(dm) )=(dw)(dm).$$  Multiplying both sides by $w$, we get $dw=0$.  Then given a homology class $[a]$ for $a \in A$, we have $[a]=[awdm]=[a][w][dm]=[0]$.

%

Thus $A$ is an acyclic algebra as required.
\end{proof}

Let $A$ be a pseudo--compact dg algebra.  
  The following arguments are essentially the ones articulated by Demazure \cite{Dema} in the (ungraded) commutative case.

We wish to show that $A$ is a direct product of local pseudo--compact dg algebras.  We know that $A$ is the inverse limit of a diagram of finite dimensional dg algebras.  Restricting to the image of $A$ in each of these finite dimensional algebras, we have that $A$ is the inverse limit of quotients by finite codimension
 dg ideals $I_i$, indexed by $i \in \mathcal{I}$ some indexing set.

Each quotient $A/I_i$ is a finite dimensional commutative dg algebra, hence, forgetting the differential, a finite product of local graded algebras. Since $du=0$ for any idempotent $u$, the differential acts in each factor, and so  
$A/I_i$ is a finite product of local dg algebras.
By replacing each of these products  with its local factors, and each morphism with its component morphisms we obtain a new diagram (with the same inverse limit) of local finite dimensional dg algebras.  Thus without loss of generality we may assume that each dg ideal $I_i$ is contained in a unique graded maximal
 ideal $M_i$ (namely the preimage of the unique graded maximal ideal in $A/I_i$).

Let $\Omega$ denote the set of graded maximal ideals of $A$ containing some $I_i$.  For each $M\in\Omega$ let $A_M$ denote the inverse limit of the subdiagram of quotients $A/I_i$ with $I_i\subset M$;
 as the inverse limit of finite-dimensional dg algebras, it is a pseudo--compact algebra.

\begin{lemma}\label{prodoverOmega}
We have an isomorphism of pseudo--compact dg algebras:
$$A\cong\prod_{M \in \Omega}A_M,$$ with the topology on the right hand side the product topology on the 
pseudo--compact
 algebras $A_M$.
\end{lemma}

\begin{proof}
An element of $A$ is a consistent assignment of congruence classes in each quotient $A/I_i$, whilst an element of $A_M$ is a consistent assignment of congruence classes to just those quotients $A/I_i$ with $I_i\subset M$.  Thus we have a natural injective algebra homomorphism $A\to\prod_{M \in \Omega}A_M$.  To see that this is surjective  we must show that any element of the product yields a consistent set of congruence classes in the $A/I_i$.

If this were not consistent then we would have
dg ideals $I_i,I_j$ contained in distinct unique graded maximal ideals $M,M'$ together with an
dg ideal $J$ containing both $I_i,I_j$.  However in that case we would have $J$ contained in a unique graded maximal ideal which must equal both $M$ and $M'$, contradicting $M \neq M'$.

Finally note that the open sets of both $A$ and $\prod_{M \in \Omega}A_M$ are generated by the preimages of subsets of the $A/I_i$.
\end{proof}

Let $M \in \Omega$.  Given $a\in A$ with component 0 in $A_M$, we have $a \in I_i$ for each $I_i \subset M$.  In particular $a \in M$.  Thus $M$ is the direct product of a graded maximal
ideal
$\hat{M}$ in $A_M$ with the remaining $A_N,\,\, N \in \Omega$.

An element $a \in A$ lies in $\hat{M}$ if its image in one of (and hence all) the $A/I_i$, (with $I_i \subset M$) lies in the unique maximal ideal.

\begin{lemma}
The algebra $A_M$ is local
pseudo--compact dg algebra.
\end{lemma}

\begin{proof}
Let $a\in A$ be a homogeneous element not in $\hat{M}$.
Then $a$
maps to a
unit in each $A/I_i$ with $I_i \subset M$.  That is it has a unique inverse in each quotient.  Let $b \in A_M$ represent this collection of congruence classes in the $A/I_i$.  Then $ab=ba=1$.
\end{proof}

Thus by  Lemma \ref{twotypes}
and Lemma \ref{prodoverOmega} we have that $A$ is a direct product of Hinich algebras and acyclic algebras.

Consider a morphism of pseudo--compact dg algebras:
\begin{eqnarray}f\colon \prod_{i \in I} A_i \to \prod_{j \in J} B_j \label{morphofpseudo}\end{eqnarray}
where the $A_i,B_j$ are all Hinich algebras or acyclic algebras.  Clearly $f$ is determined by its compositions with projections onto the factors $B_j$ for $j \in J$.  Denote these $f_j$.  Then as each $A_i, B_j$ contains no idempotents other than $0,1$ we know that each  $f_j$ factors through some ultraproduct of the $A_i$ over an ultrafilter $\mu_j$ on $I$.

\begin{lemma}
For $j \in J$ the ultrafilter $\mu_j$ is principal.
\end{lemma}

\begin{proof}
As any inverse limit of $T_1$ spaces is $T_1$, we know that the complement of $0\in B_j$ is open.  As $f$ is continuous, the preimage $f^{-1}(B_j \backslash \{0\})$ is also open.  In particular it contains a neighbourhood of $1$.  That is it contains the direct product of the $A_i$ with finitely many of the $A_i$ replaced by sets $U_i \subset A_i$, containing $1\in A_i$.

In particular $f^{-1}(B_j \backslash \{0\})$ contains an idempotent $x$ with component $1$ in finitely many factors $A_i$ and component $0$ in the rest.  As $f(x) \neq 0$ it must be the only other idempotent in $B_j$: $f(x)=1$.  Then our ultrafilter $\mu_j$ contains a finite set and is principal.
\end{proof}

Of course the ultraproduct over the principal ultrafilter $\mu_j$ is just a factor$A_{i_j}$.  That is for every $j\in J$ there exists $i_j\in I$ such that the map $f_j$ factors through the factor $A_{i_j}$.  In summary we have:

\begin{theorem}\label{pseudo}
Every pseudo--compact dg algebra is a direct product of Hinich algebras and acyclic algebras.  A morphism of pseudo--compact dg algebras $f$ (as in (\ref{morphofpseudo})) corresponds precisely to a collection of continuous homomorphisms; $f_j\colon A_{i_j}\to B_j$ for each $j \in J$.
\end{theorem}

Hinich \cite{Hini} has shown that the full subcategory of pseudo--compact dg algebras consisting of Hinich algebras is a CMC, cf. \cite{Laza1} concerning this formulation.  In the following section we extend this CMC structure to a larger full subcategory that includes acyclic algebras and show that this too is a CMC.  In \S4 we define the notion of the category of formal products of objects in a category.  From Theorem \ref{pseudo} it is clear that applying this operation to our category of Hinich algebras and acyclic algebras yields precisely the category of pseudo--compact dg algebras.

Further we show that this formal product operation takes closed model categories to closed model categories.  Thus we have that the category of pseudo--compact algebras is a CMC.  As discussed, this may be identified with $\mathcal{V}^{\rm op}$, the opposite category to $\mathcal {V}$.  Interchanging fibrations and cofibrations then implies that $\mathcal{V}$ is a CMC.

Before we proceed with this construction of a CMC structure on $\mathcal{V}$, we will briefly discuss an alternative approach.  In \cite{Smit1} the problem of inducing a CMC structure on $\mathcal{V}$ is approached by constructing an adjunction from $\mathcal{V}$ to the underlying category of dg vector spaces $\mathcal{U}$.

The functor $G\colon \mathcal{V} \to \mathcal{U}$ in this adjunction is just the forgetful functor, whilst its right adjoint $F$  is the `cofree' functor, which we will describe by explicitly giving $F^{\rm op}\colon \mathcal{U}^{\rm op} \to \mathcal{V}^{\rm op}$.  Here $F^{\rm op}$ denotes the functor which given any $V \in \mathcal{U}$, takes $V^*\mapsto (F(V))^*$.  Here we continue to regard $\mathcal{V}^{\rm op}$ as the category of pseudo--compact dg algebras over $k$ and similarly we regard $U^{\rm op}$ as the category of pseudo--compact dg vector spaces.

The following construction  is given in the  associative context in \cite[Proposition 1.10]{Getz}.  Firstly, given a finite dimensional dg vector space $V$, let $F^{\rm op}(V)=\underset{I}\varprojlim\{S(V)/I\}$, where  $S(V)$ is the free commutative dg algebra generated by $V$, and $I$ ranges over all differential ideals of finite dimensional index.

\begin{definition}
Given $V\in  U^{op}$ a pseudo--compact dg vector space with $V=\varprojlim_\gamma V_\gamma$ for finite dimensional $V_\gamma$, we define:
$$
F^{\rm op}(V)=\varprojlim_{\gamma} F^{\rm op}(V_\gamma).
$$
\end{definition}

\begin{lemma}
The functor $F$ is right adjoint to $G$.
\end{lemma}

\begin{proof}
(cf. Proof of \cite[Proposition 1.10]{Getz} for the associative case.)

We will show that the forgetful  functor $G^{\rm op}\colon \mathcal{V}^{\rm op}\to \mathcal{U}^{\rm op}$ is right adjoint to $F^{\rm op}$.  Let $A=\varprojlim_{\alpha} A_\alpha$ be a pseudo--compact algebra and let $V=\varprojlim_\gamma V_\gamma$, where as before the $V_\gamma$ are finite dimensional.  We have:
\begin{eqnarray*}
{\rm Hom}(F^{\rm op}(V),A)
&\cong&\varprojlim_{\alpha}\varinjlim_\gamma\varinjlim_I{\rm Hom}(S(V_\gamma)/I,A_\alpha)\\
&\cong& \varprojlim_{\alpha}\varinjlim_\gamma {\rm Hom}(S(V_\gamma),A_\alpha)\\
&\cong& \varprojlim_{\alpha}\varinjlim_\gamma {\rm Hom}(V_\gamma,A_\alpha)\\
&\cong& {\rm Hom}(V,A).
\end{eqnarray*}
\end{proof}

The approach of \cite{Smit, Smit1,Smit2} is to transfer the CMC structure on $\mathcal{U}$ to one on $\mathcal{V}$.  That is, a morphism $f\colon A \to B$ is defined to be a weak equivalence or cofibration precisely when the underlying morphism of dg vector spaces $G(f)$ has the corresponding property \cite[Definition 4.5]{Smit1}.

Thus the functor $G$ preserves cofibrations and acyclic cofibrations.  This is equivalent to saying that $F$ preserves fibrations and acyclic fibrations \cite[Remark 9.8]{Dwye}.   Thus $F^{\rm op}$ preserves cofibrations and acyclic cofibrations.  It follows from Ken Brown's Lemma \cite[Lemma 9.9]{Dwye} that $F^{\rm op}$ preserves weak equivalences between cofibrant objects. Since in $\mathcal{U^{\rm op}}$ all objects are cofibrant it is necessarily the case that  $F^{\rm op}$ is exact (cf \cite[Theorem 4.4]{Smit2}).

However we will show that $F^{\rm op}$ is not exact and thus that a CMC structure on $\mathcal{V}$ cannot be transferred in this way.

Let $V=\langle x,y \rangle$ be the two dimensional dg vector space over $k$  satisfying $dx=y$, with $x$ in degree 0 and $y$ in degree -1.  The map $f\colon V \to 0$ is a quasi--isomorphism.  We will demonstrate that $F^{\rm op}$ is not exact by showing that $F^{\rm op}(f)$ is not a quasi--isomorphism.    Indeed this implies that $F,G$ cannot form a Quillen adjunction for any CMC structure on $\mathcal{V}$ where $G$ preserves weak equivalences.

Let $k[|x,dx|]$ be completion of the dg algebra $S(V)=k[x,dx]$ at the dg ideal $(x)$.
\begin{proposition}
There is an isomorphism of pseudo-compact dg algebras
\[
F^{\rm op}(V)\cong \prod_{\lambda\in k}k[|x_{\lambda},dx_{\lambda}|].
\]
\end{proposition}\label{pseudodeRham}
\begin{proof}
Consider the collection $\mathcal{J}$ of principal dg ideals in $k[x,dx]$, i.e. those generated by a single polynomial $p(x)\in k[x]$; clearly this collection is cofinal among all dg ideals in $k[x,dx]$ of finite-dimensional index and so
\[
F^{\rm op}(V)\cong \varprojlim_{I\in\mathcal{J}}k[x,dx]/I.
\]
Writing $p(x)$ as a product of powers of prime ideals (which are all linear since $k$ is algebraically closed): $p(x)=(x-a_1)^{k_1}\ldots (x-a_n)^{k_n}$ we see that the  algebra $k[x]/(p(x))$ decomposes into a direct product of local algebras:
\[
k[x]/(p(x))\cong \prod_{i=1}^n k[x]/(x-a_i)^{k_i}\cong \prod_{\lambda=1}^nk[x_{\lambda}]/(x_\lambda^{k_\lambda}).
\]
We denote by $e_1,\ldots, e_n$ the corresponding set of orthogonal idempotents in $k[x]/(p(x))$; thus \[e_ik[x]/(p(x))\cong k[x]/(x-a_i)^{k_i}.\] The idempotents $e_i$ are cocycles in the dg algebra $k[x,dx]/(p(x))$ and so, give its decomposition as a direct product of dg algebras:
\[
k[x,dx]/(p(x))\cong \prod_{\lambda=1}^nk[x_{\lambda}, dx_\lambda]/(x_\lambda^{k_\lambda}).
\]
Since $\varprojlim_{k}k[x_\lambda,dx_\lambda]/(x_\lambda^{k})=k[|x_\lambda,dx_\lambda|]$, the desired statement follows.
\end{proof}

\begin{corollary}
The functor $F^{\rm op}$ is not exact.
\end{corollary}

\begin{proof}

We have that $$F^{\rm op}(V)=\prod_{\lambda \in k} k[|x_\lambda,dx_\lambda|].$$

For each $\lambda$, we have $H_0(k[|x_\lambda,dx_\lambda|])=k$.  Thus: $$H_0(F^{\rm op}(V))=\prod_{\lambda \in k} k.$$

On the other hand $F^{\rm op}(0)=k$, so $H_0(F^{\rm op}(0))=k$.  Thus $F^{\rm op}(f)$ cannot be a quasi-isomorphism, as required.
\end{proof}
\begin{remark}\label{comass}
Note that $S(V)\cong k[x,dx]$ is the de Rham algebra of $k[x]$ and $F^{\rm op}(V)$ is its pseudo-compact completion. It is, thus,  the de Rham algebra of the pseudo-compact completion $k[x]\wht$ of $k[x]$; and we showed that it is a direct product of copies of the de Rham algebra of $k[|x|]$. In contrast, the corresponding procedure applied to the free associative algebra $k\langle x,dx\rangle$ leads to the pseudo-compact version of the algebra of noncommutative forms on $k[x]\wht$. The cohomology of the complex of noncommutative forms is always isomorphic to $k$ sitting in degree zero, cf. for example \cite{Ginz}, 11.4. Technically, the arguments in the proof of Proposition \ref{pseudodeRham} fail because the idempotents $e_i$ are no longer cocycles in the algebra of noncommutative forms.  It follows that the cofree coalgebra functor is exact in the noncommutative context. This was also proved by Getzler and Goerss, \cite{Getz}, Theorem 2.1 by constructing an explicit chain homotopy.
\end{remark}

\section{The Extended Hinich Category}\label{extsec}

Let $\mathcal{E}$ denote the full subcategory of the category of  pseudo--compact dg algebras, whose objects are precisely the local pseudo--compact dg algebras.  As we have seen, these algebras are either Hinich algebras or acyclic algebras.  The goal of this section is to show that $\mathcal{E}$ has the structure of a CMC.

As $d1=0$, every acyclic algebra contains an element $x$ with $dx=1$.  Moreover, $x$ may chosen to be of degree 1, so an acyclic  algebra contains the algebra $\Lambda(x)= k[x]/x^2$, with differential given by $dx=1$.

\begin{lemma}\label{acyc} Let $A$ be an acyclic algebra.  Then $A=A^0 \otimes \Lambda$, where $A^0$ is the subalgebra of closed elements.
\end{lemma}

\begin{proof}
Any element $w \in A$ may be written $w=(w-xdw)+xdw$, where $w-xdw, dw$ are closed. Conversely, given $w=a+xb$, with $a,b$ closed, we have $b=dw$ and $a=w-xdw$.
\end{proof}

\begin{definition}
Let $A$ be an object in $\mathcal{E}$, with unique maximal ideal $M$.   The full Hinich subalgebra $A^H$ consists of all elements $a$, such that $da\in M$.
\end{definition}

Note that if $A$ is a Hinich algebra then $A^H=A$.  Conversely if $A$ is an acyclic algebra, then $A^H$ has codimension 1 in $A$.

\begin{lemma}
Both $A^H$ and $A^0$ are Hinich algebras.
\end{lemma}

\begin{proof}
Again let $M$ denote the maximal ideal in $A$.  Let $M^H, M^0$ denote the intersections of $M$ with $A^H$ and $A^0$ respectively.

If $x\in A^H$ is not an element of $M^H$, then we have $y\in A$ such that $xy=1$.  We need to show that $y\in A^H$, in order to deduce that $x$ is a unit in $A^H$, and hence that $A^H$ is local.

We have $0=d(xy)=(dx)y\pm x(dy)$.  Thus $x(dy)\in M$.  In particular $yx(dy)=\pm dy\in M$ and $y \in A^H$ as required.

Similarly, given $x \in A^0$ with $x$ not an element of $M^0$, we have $y\in A$ with $xy=1$ and we need to show that $y \in A^0$, to deduce that $x$ is a unit and that $A^0$ is local.

Again we have $0= (dx)y\pm x(dy)=\pm x(dy)$.  Hence, as before $dy=\pm yx(dy)=0$ and $y\in A^0$ as required.
\end{proof}

\begin{lemma} \label{morph}
Let $f\colon A \to B$ be a morphism in $\mathcal{E}$.  Then $f$ restricts to a morphism $f^H\colon A^H \to B^H$.  Conversely given $x\notin A^H$, we have $f(x)\notin B^H$.

\end{lemma}

\begin{proof}
If $x\in A$ satisfies $dx\in M$, then consider the image of $f(dx)$ in $B/N \cong k$, where $N$ is the maximal ideal in $B$.  This image must be 0, as any surjective map from $A$ to a field will have kernel $M$.  Thus $d(f(x))=f(dx)\in N$ and $f(x)\in B^H$.

For the converse, note that if $dx$ is a unit, then $d(f(x))=f(dx)$ is also a unit.
\end{proof}

In particular, Lemma \ref{morph} implies that there are no morphisms in $\mathcal{E}$ from an acyclic algebra to a Hinich algebra, and any morphism from a Hinich algebra $A$ to an acyclic algebra $B$ must factor through $B^H$.

\begin{lemma} \label{colim} The category $\mathcal{E}$ contains all (small) colimits.
\end{lemma}

\begin{proof}
It suffices to show that $\mathcal{E}$ contains all coproducts and coequalisers.  The coproduct in $\mathcal{E}$ is just the tensor product.  The coequaliser of two maps $f\colon A\to B$ is just the quotient of $B$ by the ideal generated by elements of the form $f(a)-g(a),\, a \in A$.
\end{proof}

Given a diagram $D$ in $\mathcal{E}$, let $D^0,D^H$ denote the corresponding diagrams of subalgebras.  

We will construct a limit for $D$ by considering separately the case where there is a cone $\Lambda(x) \to D$ and the case where there is not.  Note that by Lemma \ref{morph}, if even one of the objects in $D$ is a Hinich algebra, then we are in the second case.

\begin{lemma}\label{nomap}
Suppose there is no cone $\Lambda(x) \to D$.  Then we have $\varprojlim(D)=\varprojlim(D^H)$, the limit in the usual Hinich category.
\end{lemma}

\begin{proof}
Given a cone $X \stackrel f\to  D$, we know that $X$ is a Hinich algebra, as otherwise we would have a cone $\Lambda(x) \to X \stackrel f\to D$.  By Lemma \ref{morph} we have that $f$ factorizes $X \stackrel {f'}\to D^H \stackrel\iota \to D$, where $\iota$ denotes the inclusion on each object of $D$.  Then $f'$ factorizes uniquely through the cone $\varprojlim(D^H) \to D^H$:

$$
\xymatrix{X\ar@/_/[ddr]_f \ar[dr]_{f'}\ar@{.>}[r] & \varprojlim(D^H)\ar[d]\\&D^H\ar[d]_{\iota}\\&D
}
$$

As the component maps of $\iota$ are injective, this is the unique factorization of $f$ through the cone $\varprojlim{D^H}\to D^H \stackrel \iota \to D$.
\end{proof}

Suppose there is a cone $g\colon\Lambda(x) \to D$.  Then let $L=\varprojlim(D^0) \otimes \Lambda(x)$, and let $l\colon L \to D$ be the cone induced by $g$ and $\varprojlim{D^0}\to D^0 \stackrel \iota \to D$, where $\iota$ denotes the inclusion on each object of $D$:

$$
\xymatrix{\varprojlim{D^0} \ar[d]\ar[r] & L\ar[d]^l& \Lambda(x)\ar[dl]^g\ar[l]\\D^0\ar[r]^\iota&D
}
$$

\begin{lemma}\label{vani} Let $A$ be a Hinich algebra with vanishing differential and let $f\colon A \to D$ be a cone.  Then $f$ factors uniquely through $l$.
\end{lemma}

\begin{proof}
As $f$ factors through $D^0$ and hence $\varprojlim{D^0}$, we have that the following diagram commutes, for some $f' \colon A \to \varprojlim{D^0}$:
$$
\xymatrix{A\ar[dr]^{f'}\ar@/_4pc/[ddrr]_f\\
&\varprojlim{D^0} \ar[d]\ar[r] & L\ar[d]^l& \Lambda(x)\ar[dl]^g\ar[l]\\&D^0\ar[r]^\iota&D
}
$$
Now any different morphism $f''\colon A \to L$ making the diagram commute, must factor through $\varprojlim{D^0}$, as $L^0=\varprojlim{D^0}$.  As the components of $\iota$ are injective, the induced map $A \to \varprojlim{D^0}$ must in fact be $f'$, and $f''$ is not a different morphism after all.
\end{proof}

\begin{lemma}\label{lamby}
Let $f\colon \Lambda(y) \to D$ be a cone.  Then $f$ factors uniquely through $l$.
\end{lemma}

\begin{proof}
The cones $f,g$ from $\Lambda(x),\, \Lambda(y)$ induce a cone from the coproduct $\Lambda(x,y)$.  The restriction of this to $\Lambda(x,y)^0$ factors through $\varprojlim{D^0}$, by Lemma \ref{vani}.  We obtain the following commuting diagram:

$$
\xymatrix{&&\Lambda(x,y)^0\ar[r]\ar[dll]&\Lambda(x,y)\ar@/___9pc/[ddll]\\
\varprojlim{D^0} \ar[d]\ar[r] & L\ar[d]^l& \Lambda(x)\ar[ur]\ar[dl]^g\ar[l]& \Lambda(y)\ar[u]\ar@/___/[dll]^f\\
D^0\ar[r]^\iota&D
}
$$
\vspace{1.52cm}

Let $\alpha\in\varprojlim{D^0}$ denote the image of $y-x\in \Lambda(x,y)^0$.  We define a map $h\colon \Lambda(y) \to L$ by $h\colon y \mapsto \alpha\otimes 1 +1\otimes x$.

Then $f=lh$ as $lh(y)=(f(y)-g(x))+g(x)=f(y)$.

Any other morphism $h'\colon \Lambda(y) \to L$ factorizing $f$ would map $y\mapsto h(y)+\beta\otimes 1$, for some $\beta\in\varprojlim {D^0}$ satisfying $l(\beta \otimes 1)=0$.  As the component maps of $\iota$ are injective, we have that the image of $\beta$ in each object of $D^0$ is 0.  Therefore $\beta=0$ and $h'=h$.
\end{proof}

\begin{lemma}\label{ifmap}
We have $L=\varprojlim{D}$.
\end{lemma}

\begin{proof}
Firstly, let $A$ be an acyclic algebra and let $f\colon A \to D$ be a cone.  By Lemma \ref{acyc} we know that $A$ is the coproduct of $A^0$ and $\Lambda(y)$.  The restrictions of $f$ to $A^0$ and $\Lambda(y)$, factor through $l$ uniquely, by Lemmas \ref{vani} and \ref{lamby} respectively.  These induce a map $h\colon A \to L$ which is then the unique factorization of $f$ through $l$: 
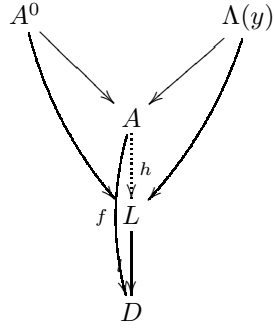
$$
\xymatrix{A^0\ar[dr]\ar@/_/[ddr]&&\Lambda(y)\ar[dl]\ar@/___/[ddl]\\
&A\ar@/_/[dd]_f\ar@{.>}[d]^h\\
&L\ar[d]_l\\
&D
}
$$
Now let $A$ be a Hinich algebra and let $f\colon A \to D$ be a cone.  We have a cone $A \otimes \Lambda(x) \to D$ induced by $f,g$, which factors through $l$ as $A \otimes \Lambda(x)$ is acyclic:
$$
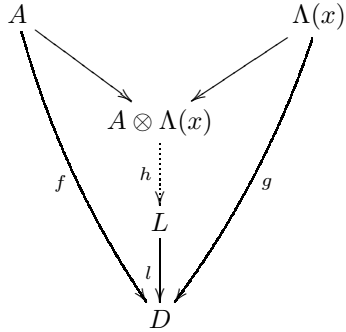
\xymatrix{A\ar[dr]\ar@/_/[dddr]_f&&\Lambda(x)\ar[dl]\ar@/___/[dddl]^g\\
&A\otimes\Lambda(x)\ar@{.>}[d]_h\\
&L\ar[d]_l\\
&D
}
$$
Let $f'\colon A \to L$ be a different map which factorises $f$ through $l$.  Together with the natural inclusion $\Lambda(x) \to L$, there is induced a map $h'\colon A\otimes \Lambda(x) \to L$.  Now $h'$ also factorizes the cone $A \otimes \Lambda(x) \to D$.  Thus by the uniqueness of this factorization $h'=h$ and $f'$ is not a different map after all.
\end{proof}

Thus by Lemma \ref{colim} we know $\mathcal{E}$ contains all (small) colimits and by Lemmas \ref{nomap} and \ref{ifmap} we know $\mathcal{E}$ contains all (small) limits.  Note that by Lemma \ref{nomap} the product $A \times k= A^H$, for any algebra $A$ in $\mathcal{E}$.  Whilst $k$ was both an initial and a terminal object in the Hinich category, in $\mathcal{E}$, it is just an initial object, and we have a new terminal object:

\begin{lemma}
The terminal object in $\mathcal{E}$ is $\Lambda(x)$.
\end{lemma}

\begin{proof}
From any Hinich algebra there is a unique morphism to $\Lambda(x)$, factoring through $k \to \Lambda(x)$ (by Lemma \ref{morph}).  Any acyclic algebra is the coproduct of a Hinich algebra with $\Lambda(x)$, so it suffices to note that the identity is the unique morphism $\Lambda(x) \to \Lambda(x)$.  This follows from the fact that any such morphism must map $x$ to an element of the maximal ideal, whose derivative is $1$.
\end{proof}

\begin{definition}\label{cmc}
We define \emph{fibrations}, \emph{cofibrations} and \emph{weak equivalences}{\rm :}

\bigskip
\noindent
--A map $f\colon A \to B$ is a \emph{fibration} precisely when $B^H$ is contained in the image of $f$.

\bigskip
\noindent
--A \emph{cofibration} is a retraction of  a morphism in the class $\mathcal{C}$, consisting of the tensor product of cofibrations in the Hinich category with{\rm:}

\bigskip
\noindent i) The identity $1_k\colon k \to k$,

\noindent ii) The identity $1_{\Lambda(x)}\colon \Lambda(x) \to \Lambda(x)$,

\noindent iii) The natural inclusion $k \to \Lambda(x)$.

\bigskip
\noindent --A \emph{weak equivalence} is either a weak equivalence (in the Hinich category) between Hinich algebras, or any morphism between acyclic algebras.
\end{definition}

Note the fibrations between Hinich algebras are precisely the surjective maps, which are also the fibrations in the Hinich category.  (Indeed the fibrations between acyclic algebras are also precisely the surjective maps).  Further the weak equivalences and cofibrations between Hinich algebras are precisely the same as the ones in the Hinich category.  Thus this CMC structure extends the one given in \cite{Hini}.

Also note that every object in $\mathcal{E}$ is fibrant, as the image of any map to $\Lambda(x)$ must contain $k \subset \Lambda(x)$.

It is clear from Definition \ref{cmc} and the fact that the Hinich category is a CMC that fibrations, cofibrations and weak equivalences are closed under retraction, and each contain all identity maps.  Further, weak equivalences must satisfy the 2 of 3 rule and fibrations are closed under composition.  To verify that we have indeed defined a CMC, it remains to check that the various lifting properties and factorizations hold (from which it will follow that cofibrations are closed under composition).

\begin{lemma}  \label{acycacyc}Let $B_1, B_2$ be acyclic algebras.  Then any morphism $f\colon B_1^H \to B_2^H$ is a weak equivalence of Hinich algebras.
\end{lemma}

\begin{proof}
By the 2 of 3 rule for weak equivalences, it suffices to check that maps $B^H \to k$ are weak equivalences in the Hinich category.  Let $I$ denote the maximal ideal in $B^0$.  From Lemma \ref{acyc} we know that $B^H = k\oplus I[x]/x^2$, with differential given by $d(a+xb)=b$ for $a,b \in I$.

Clearly the map  $k\oplus I[x]/x^2 \to k$ is a quasi--isomorphism.  Weak equivalences in the Hinich category are essentially filtered quasi-isomorphisms. Take the filtration on $I[x]/x^2$ by the powers of $I$. This is an admissible filtration, and the given map $k\oplus I[x]/x^2\to k$ respects it and induces a quasi-isomorphism on the associated grading.
\end{proof}

\begin{lemma} \label{llpcof}
Cofibrations have the left lifting property with respect to acyclic fibrations.
\end{lemma}

\begin{proof}
As the left lifting property with respect to a given map is closed under retraction, it suffices to verify it for each of the three classes (i), (ii), (iii).

Given a commutative square with a Hinich cofibration $f\colon A_1 \to A_2$ on the left and an acyclic fibration $g\colon B_1 \to B_2$ on the right, we may factorize the horizontal maps through the full Hinich subalgebras of $B_1, B_2$.  The restriction of $g$ to these is still a fibration in the Hinich category, and a weak equivalence (by Lemma \ref{acycacyc} in the  case $B_1, B_2$ are acyclic), so we obtain a lifting as follows:
\begin{eqnarray}\label{Hinifactor}
\xymatrix{
A_1\ar[d]_f\ar[r]&B_1^H\ar[d]\ar[r]&B_1\ar[d]^{g}\\
A_2\ar[r]\ar@{.>}[ur]&B_2^H\ar[r]&B_2
}
\end{eqnarray}
Composed with the horizontal map $B_1^H \to B_1$ this gives a lifting of the original square.

As the left lifting property is closed under taking coproducts, we have that the left lifting property for acyclic fibrations holds for maps in the classes (i) and (ii).  To show that it also holds for maps in the class (iii), it remains to show that it holds for the map $k\to \Lambda(x)$.

Let $g\colon B_1 \to B_2$ be an acyclic fibration and suppose we have a commutative square:
$$
\xymatrix{
k\ar[d]\ar[r]^{i_1}&B_1\ar[d]^{g}\\
\Lambda(x)\ar[r]^{i_2}&B_2
}
$$
Let $z\in B_1$ be a degree 1 element satisfying $g(z)=i_2(1\otimes x)$.   Then $dz=1+m$ for some degree 0 element $m\in {\rm ker}(g)$ with $dm=0$.

Let $y\in B_1$ be a degree 1 element satisfying $dy=1$.  Then we have that $d(z-ym)=1$.  Thus we may define a lifting $h\colon \Lambda(x) \to B_1$ by $h\colon x \mapsto z-ym$.
\end{proof}

\begin{lemma}
Acyclic cofibrations have the left lifting property with respect to fibrations.
\end{lemma}

\begin{proof}

Suppose we have a commutative square:
$$
\xymatrix{
A_1\ar[d]_f\ar[r]&B_1\ar[d]^{g}\\
A_2\ar[r]&B_2
}
$$
with $f$ an acyclic cofibration and $g$ a fibration.

If $A_1,A_2$ are Hinich algebras, then we may factorize the horizontal maps though full Hinich subalgebras and obtain a lifting as in (\ref{Hinifactor}).

On the other hand, if $A_1,A_2$ are acyclic algebras, then so are $B_1,B_2$.  Thus $g$ is an acyclic fibration and we have a lifting by Lemma \ref{llpcof}.
\end{proof}

\begin{lemma}
Let $f\colon A \to B$ be a morphism in $\mathcal{E}$.  Then $f$ may factorized as{\rm :}

\noindent i) an acyclic cofibration followed by a fibration,

\noindent ii) a cofibration followed by an acyclic fibration.
\end{lemma}

\begin{proof}
In both cases it will suffice to take cofibrations from the class $\mathcal{C}$.

If $A,B$ are Hinich algebras, then both factorizations follow from the fact that the Hinich category is a CMC.  Note the cofibrations used here are of type (i).

Next suppose $A,B$ are both acyclic algebras.  By Lemma \ref{acyc}  we have $A=A^0\otimes \Lambda(x)$.  Letting $x$ also denote $f(x)$ we have that $f=f^0 \otimes 1_{\Lambda(x)}:A^0\otimes\Lambda(x)\to B^0 \otimes \Lambda(x)$, where $f^0\colon A^0 \to B^0$ is the restriction of $f$.

Now as $f^0$ is a map of Hinich algebras we may factorize it as $f^0=ji$ for an object $C$, a cofibration $i\colon A^0\to C$ and a fibration $j\colon C \to B^0$.  Then $i \otimes 1_{\Lambda(X)}$ is a cofibration of type (ii) and $j \otimes 1_{\Lambda(X)}$ is still surjective, hence a fibration.  Both are of course weak equivalences as they are maps between acyclic algebras.  Thus both factorizations of $f$ are given by:
$$
\xymatrix{A \ar[r]^{i \otimes 1_{\Lambda(X)}\quad\,\,}&C\otimes\Lambda(x) \ar[r]^{\qquad j \otimes 1_{\Lambda(X)}}& B}
$$

Finally, let $A$ be a Hinich algebra and $B$ an acyclic algebra.  We may factorize $f$ through the map $f^H\colon A \to B^H$, followed by the inclusion $\iota\colon B^H \to B$.  Then we may factorize $f^H=ji$, for an object $C$, an acyclic cofibration $i\colon A \to C$, and a fibration $j\colon C \to B^H$.  Thus $f$ factorises as:
$$
\xymatrix{A \ar[r]^{i}&C \ar[r]^{j}& B^H\ar[r]^\iota&B}
$$
As $j$ is a fibration between Hinich algebras it is surjective, and $\iota j$ is a fibration.  Thus $f=(\iota j) i$ gives us the first factorization (i).

From Lemma \ref{acyc} we know that we have some morphism $g\colon \Lambda(x) \to B$.   Together with $\iota j$ this induces a morphism $h\colon C\otimes \Lambda(x) \to B$.  This map is a surjective map of acyclic algebras, hence an acyclic  fibration.  Let $c\colon C \to C\otimes \Lambda(x)$ denote the natural inclusion.  We have that $f$ factorizes as:
$$
\xymatrix{A \ar[r]^{i}&C \ar[r]^{c\qquad}&C\otimes \Lambda(x)\ar[r]^{\qquad h}& B}
$$
Now $ci$ is a cofibration of type (iii).  Thus $f=h(ci)$ gives us the second factorization (ii).

\end{proof}

\begin{lemma}
Cofibrations are closed under composition.
\end{lemma}

\begin{proof}
Note that compositions of cofibrations have the left lifting property for acyclic fibrations by Lemma \ref{llpcof}.  Let $f\colon A \to B$ be such a composition, and let $f=ji$ be its factorization into a morphism $i \in \mathcal{C}$ followed by an acyclic fibration $j\colon C \to B$.  We may pick a morphism $h$ to make the following diagram commute:
$$
\xymatrix{
A\ar[d]_f\ar[r]^i&C\ar[d]^j\\
B\ar[r]_{1_B}\ar@{.>}[ur]_h&B&
} $$

Then $f$ is a retract of the morphism $i\in\mathcal{C}$:
$$
\xymatrix{
A\ar[d]_f\ar[r]^{1_A}&A\ar[d]^i\ar[r]^{1_A}&A\ar[d]^f\\
B\ar[r]_h&C\ar[r]_j&B
} $$
\end{proof}

Thus $\mathcal{E}$ contains limits and colimits.  We conclude:

\begin{theorem} \label{section3mainresult}
The category $\mathcal{E}$  is a CMC, with fibrations, cofibrations and weak equivalences as in Definition \ref{cmc}.
\end{theorem}

The demand that there are no weak equivalences between a Hinich algebra and an acyclic algebra was quite natural to make, as such maps are never quasi--isomorphisms.  Similarly it was natural to demand that all maps between acyclic algebras be weak equivalences.

Once the weak equivalences are fixed, the extension of the CMC structure on Hinich algebras to a CMC structure on $\mathcal{E}$ is completely determined by two conditions.  The first is that surjective maps are fibrations (as they were in the Hinich algebra).  The second is that the terminal object is cofibrant.

\begin{proposition}
The CMC structure on $\mathcal{E}$ given in Definition \ref{cmc}, is the unique extension of the one on Hinich algebras given in \cite{Hini} with{\rm:}

\bigskip
\noindent i) weak equivalences precisely those given in \ref{cmc},

\bigskip
\noindent ii)  all surjective maps being fibrations,

\bigskip
\noindent iii) $\Lambda(x)$ cofibrant.
\end{proposition}

\begin{proof}
Consider any such CMC structure.  The map $\gamma\colon k \to \Lambda(x)$ has the right lifting property for all maps between Hinich algebras, and all maps between acyclic algebras.  Thus it must be a fibration.  For any algebra $A$, the inclusion $A^H \to A$ is the product of the identity $1_A$ with $\gamma$ and hence a fibration.  Thus all the fibrations of Definition \ref{cmc} are still fibrations in our new CMC structure.

Clearly $1_{\Lambda(x)}$ is a cofibration and as $\Lambda(x)$ is cofibrant we have $\gamma$ also a cofibration. Thus for any cofibration between Hinich algebras $f\colon A \to B$, the coproduct of $f$ with $1_{\Lambda(x)}$ or $\gamma$ must also be a cofibration.   Thus all the cofibrations of Definition \ref{cmc} are still cofibrations in our new CMC structure.

Our new CMC structure has the same weak equivalences, and all the fibrations and cofibrations of Definition \ref{cmc}.  Thus it cannot have any additional fibrations or cofibrations and must in fact be the same structure.
\end{proof}

Finally we note that the CMC structure on $\mathcal{E}$ is related to the CMC structures on Hinich algebras via the following Quillen adjunction:

\bigskip
-- Let $E$ be the inclusion functor of the category of Hinich algebras in $\mathcal{E}$.

\bigskip
-- Let $H$ be the functor sending an algebra $A\in \mathcal{E}$ to $A^H$, and sending a morphism $f\colon A \to B$ in $\mathcal{E}$ to the restriction $f^H\colon A^H \to B^H$.

\bigskip
We have a natural equivalence $\epsilon \colon EH\to 1_E$, where for any $A\in \mathcal{E}$, the map $\epsilon_A\colon A^H \to A$ is just the natural inclusion.   We have a natural isomorphism $\eta$ from $HE$ to the identity functor on Hinich algebras, which is simply the identity on each Hinich algebra.  Thus $E$ is left adjoint to $H$: $E \dashv H$.

\begin{proposition} \label{quilladjEH}
The pair $E,H$ form a Quillen adjunction.
\end{proposition}

\begin{proof}
It suffices to note that $E$ preserves cofibrations and weak equivalences (as well as fibrations).
\end{proof}

\section{Formal categories of products}\label{prodsec}

In this section we will show that if a category $\mathcal{C}$ is a CMC, then so is the category of formal products of objects in $\mathcal{C}$.

\begin{definition} Let {\rm Prod}$(\mathcal{C})$ denote the category whose objects are maps from an indexing set to the objects of $\mathcal{C}$.  We write $\prod_{i\in I} A_i$ for the object corresponding to the  map sending $i \mapsto A_i$.

A morphism in {\rm Prod}$(\mathcal{C})${\rm ;} $$f\colon \prod_{i\in I} A_i \to \prod_{j\in J} B_j,$$ is a map $J \to I$, sending $j \to i_j$, together with a map sending each $j \in J$ to a morphism $f_j\colon A_{i_j} \to B_j$.  We call the $f_j$ the components of $f$.

The composition of $f$ with a map $g\colon \prod_{j\in J} B_j\to \prod_{k \in K} C_k$ is given by{\rm:}
$$
(gf)_k=g_kf_{j_k}
$$
The identity map on $\prod_{i \in I} A_i$ has components $1_{A_i}$ for $i \in I$.
\end{definition}

Note that the indexing set for an object may be empty.  We denote this object $0$.  Observe that it is the terminal object in our category;  given any other object, there is a unique morphism from it to $0$, which will have no components.  The only morphism from $0$ is the idenity map to itself.

We can also define the category of formal coproducts in $\mathcal{C}$ by \[\left(\operatorname{coProd} \mathcal{C}\right)^{\rm op}:=\left(\operatorname{Prod}{\mathcal C}^{\rm op}\right)\]
where the superscript $\rm op$ stands to indicate the opposite category. Clearly, all our results will also be true for the categories of formal coproducts, in particular $\operatorname{coProd} \mathcal{C}$ will be a CMC whenever $\mathcal{C}$ is a CMC.
\begin{remark}
To get some intuition for the category $\operatorname{Prod}\mathcal C$ or  $\operatorname{coProd}\mathcal C$ note that if $\mathcal C$ is the category of connected topological spaces then $\operatorname{coProd}\mathcal C$ is the category of all topological spaces. The categories of formal (co)products (more precisely, their variants consisting of \emph{finite} (co)products) were used in \cite{Laza} to construct disconnected rational homotopy theory.
\end{remark}
Note that $\operatorname{ Prod}$ of the category of local $k$--algebras is not the full category of $k$--algebras (note that a morphism $k^I \to k$ need not correspond to the inclusion of a point in the discrete set $I$, but rather the inclusion of a point in $\beta I$, the Stone--\u{C}ech compactification of $I$).

However we need not concern ourselves with this technicality for coalgebras, as from Theorem \ref{pseudo} we have that:

\begin{lemma}\label{coalg} The opposite category to  counital cocommutative dg coalgebras $\mathcal{V}^{\rm op}$ is precisely ${\rm  Prod}(\mathcal{E})$.
\end{lemma}

Returning to the general case $\mathcal{C}$, we must first show that {\rm Prod}$(\mathcal{C})$ contains all limits and colimits.  Let $D\colon {\mathcal{A}}\to {\rm Prod}(\mathcal{C})$ be a diagram.  Let $\bar{D}\colon \mathcal{A}^{\rm op} \to {\rm SET}$ denote the underlying diagram of indexing sets.  Let $L,C$ denote the limit and colimit of $\bar{D}$ respectively.  We have $$C=\coprod_{A \in \mathcal{A}}\bar{D}(A)/\sim,$$ where the relation is generated by $i \sim \bar{D}(f)(i)$ for $f$ a morphism in $\mathcal{A}$.

Thus each $c \in C$ may be regarded as a category, whose objects are the elements of the equivalence class $c$ and whose morphisms are arrows $\bar{D}(f)(i)\to i$, for morphisms $f$ in $\mathcal{A}$.  We have a diagram $D_c\colon c \to \mathcal{C}$ sending each object and arrow in $c$ to the corresponding object and morphism in $\mathcal{C}$.  Let $L_c=\varprojlim{D_c}$ and let $$L_D=\prod_{c\in C} L_c.$$

\begin{lemma}\label{prodlim}
We have $\varprojlim{D}=L_D$.
\end{lemma}

\begin{proof}
Given some $A\in\mathcal{A}$, let $$D(A)=\prod_{i\in I}A_i.$$  Each $i\in I$ is an element of a unique $c \in C$, and we have a map ${f^A}_i\colon L_c \to A_i$.  Thus we have a morphism $$f^A\colon L_D \to D(A),$$ for each $A\in \mathcal{A}$.  This gives us a cone $f\colon L_D \to D$.

Now consider a cone $g\colon X=\prod_{j \in J}X_j \to D$.   For each $c \in C$ we have that $g$ maps all the elements of $c$ to a unique $j \in J$.  We thus have a cone $g_c\colon X_j \to D_c$, which factors uniquely through a morphism $h_c\colon X_j \to L_c$, so the following diagram commutes:
$$\xymatrix{
X_j\ar[dr]_{g_c}\ar@{.>}[r]^{h_c}&L_c\ar[d]\\&D_c
}$$
The resulting morphism $h\colon X \to L_D$, gives the unique factorization of $g$ through the cone $f$:
$$\xymatrix{
X\ar[dr]_{g}\ar@{.>}[r]^{h}&L_D\ar[d]^f\\&D
}$$
\end{proof}

\begin{corollary}
The product over $j \in J$ of objects $\prod_{i \in I_j} A_i$ is{\rm:}
$$
\prod_{i \in \coprod_{j \in J} I_j} A_i
$$
\end{corollary}

Next we consider colimits.  Again given an object $A\in \mathcal{A}$, write $D(A)=\prod_{i\in I} A_i$.  An element $l \in L$ assigns an index $i_l$ to to each $A\in\mathcal{A}$.  Thus $l$ defines a diagram $D_l\colon \mathcal{A} \to \mathcal{C}$, sending an object $A$ to $A_{i_l}$, and a morphism $f\colon A \to A'$ to the component of $D(f)$: $f_l\colon D_l(A) \to D_l(A')$.

Let $C_l=\varinjlim{D_l}$ and let:$$
C_D=\prod_{l \in L} C_l$$

\begin{lemma} \label{prodcolim}
We have $\varinjlim{D}=C_D$.
\end{lemma}

\begin{proof}
For each $A \in \mathcal{A}$ and $l \in L$ we have a morphism ${f^A}_l\colon D_l(A) \to C_l$ as $C_l$ is the colimit of $D_l$.  Thus we have a morphism $f^A\colon D(A)\to C_D$.  Collectively these morphisms form a co--cone $f\colon D \to C_D$.

Now consider a co--cone: $$g\colon D\to X=\prod_{j \in J} X_j$$
Each $j \in J$ is mapped to some $l \in L$ by $g$, and we have a co--cone $g_j\colon D_l \to X_j$.  This factors uniquely through a map $h_j\colon C_l \to X_j$, making the following diagram commute:
$$\xymatrix{
D_l\ar[d]\ar[dr]^{g_j}\\
C_l\ar@{.>}[r]_{h_j}&X_j
}$$
Thus we have a morphism $h\colon C_D \to X$ which gives the unique factorization of $g$ through $f$:
$$\xymatrix{
D\ar[d]_f\ar[dr]^{g}\\
C_D\ar@{.>}[r]_{h}&X
}$$
\end{proof}

\begin{corollary}
The coproduct distributes over products.  That is{\rm:}
$$
(\prod_{i\in I} A_i)\,\,\sqcup\,\, (\prod_{j \in J} B_j) \quad= \prod_{(i,j)\in I\times J} (A_i \sqcup B_j)
$$
\end{corollary}

Consider a morphism:  \begin{eqnarray}f\colon \prod_{i\in I} A_i \to \prod_{j\in J} B_j \label{mapclass}\end{eqnarray}  For each $i\in I$ let $B^i$ denote the product in $\mathcal{C}$ of  the $B_j$ satisfying $i_j=i$.  The maps $f_j \colon A_i \to B_j$ factor uniquely through a map  $f^i\colon A_i \to B^i$.

\bigskip
We now define a CMC structure on Prod$(\mathcal{C})$.

\begin{definition}\label{struct}
The map $f$ (from (\ref{mapclass})) is:

\bigskip\noindent
-- a cofibration precisely when for each $j\in J$ we have $f_j$ is a cofibration in $\mathcal{C}$,

\bigskip\noindent
-- a fibration precisely when for each $i\in I$ we have $f^i$ is a fibration in $\mathcal{C}$,

\bigskip\noindent
-- a weak equivalence precisely when the map $I \to J$ (induced by $f$) is a bijection, and each $f_j$, $j \in J$ (or equivalently each $f^i$, $i \in I$)  is a weak equivalence in $\mathcal{C}$.

\end{definition}
\begin{theorem}\label{thm:prod}
If $\mathcal C$ is a CMC then Definition \ref{struct} determines a CMC structure on Prod$(\mathcal{C})$.
\end{theorem}
The proof of Theorem \label{thm:prod} consists of a succession of lemmas below. Clearly the identity map on any object is a cofibration, fibration and weak equivalence.  Also it is clear that cofibrations and weak equivalences are each closed under composition.

\begin{lemma}
Fibrations are closed under composition.
\end{lemma}

\begin{proof}
Suppose we have fibrations:$$f\colon \prod_{i\in I} A_i \to \prod_{j\in J} B_j,\qquad g\colon  \prod_{j\in J} B_j \to \prod_{k\in K} C_k$$
Then for each $j \in J$, we have a fibration: $g^j\colon B_j \to C^j$.  Given $i \in I$, let $g^i\colon B^i \to C^i$ denote the product (in $\mathcal{C}$) of all the $g^j$ which satisfy $i_j=i$.  Fibrations in a CMC are closed under products, so we know that $g^i$ is a fibration for each $i\in I$. Recall that $f^i$ is also a fibration for $i \in I$.

Now consider any $k\in K$ and let $j=j_k,\, i=i_j$.  The following diagram commutes (where the vertical arrows denote the natural projections from products):
$$\xymatrix{A_i\ar@/_2pc/[ddrr]_{(gf)_k}\ar[r]^{f^i}&B^i\ar[r]^{g^i}\ar[d]&C^i\ar[d]\\
&B_j\ar[r]^{g^j}&C^j\ar[d]\\
&&C_k
}$$
Thus given $i\in I$ we have that $(gf)^i=g^if^i$ is a fibration in $\mathcal{C}$.  Thus $gf$ is a fibration.
\end{proof}

Consider morphisms $f\colon A \to B$, $g\colon B \to C$ in Prod$(\mathcal{C})$.   Let the objects $A,B,C$ be indexed by sets $I,J,K$ respectively.

\begin{lemma}
If any two of $f$, $g$, $gf$ are weak equivalences, then so is the third.
\end{lemma}

\begin{proof}
If any two of $f$, $g$, $gf$ induce bijections on sets, then so does the third.   For any $k \in K$ we have that two out $g_k$, $f_{j_k}$ and $g_k f_{j_k}$ are weak equivalences in $\mathcal{C}$.  Hence we may conclude that the third is as well.
\end{proof}

Now suppose we have a commutative diagram in Prod$(\mathcal{C})$:
\begin{eqnarray}\xymatrix{
A \ar[d]_f \ar[r]^\iota&B\ar[d]_g\ar[r]^\rho&A\ar[d]_f\\
A'\ar[r]^{\iota'}&B'\ar[r]^{\rho'}&A'
}\label{retract} \end{eqnarray}
where $\rho\iota=1_A$, $\rho'\iota'=1_{A'}$ and $A,B,A',B'$ are indexed by sets $I,J,I',J'$ respectively.

\begin{lemma}
If $g$ is a cofibration then so is $f$.  If $g$ is a weak equivalence then so is $f$.
\end{lemma}

\begin{proof}
In the category of sets, retractions of bijections are bijections.  Thus if $g$ induces a bijection on indexing sets then so does $f$.  Now given any $i' \in I'$ we have a commutative diagram in $\mathcal{C}$:
$$\xymatrix{
A_i \ar[d]_{f_{i'}} \ar[r]^{\iota_j}&B_j\ar[d]_{g_{j'}}\ar[r]^{\rho_i}&A_i\ar[d]_{f_{i'}}\\
A'_{i'}\ar[r]^{\iota'_{j'}}&B'_{j'}\ar[r]^{\rho'_{i'}}&A'_{i'}
}$$
We observe that $f_{i'}$ is a retraction of $g_{j'}$.  Thus $f_{i'}$ is a cofibration if $g_{j'}$ is and $f_{i'}$ is a weak equivalence if $g_{j'}$ is.
\end{proof}

\begin{lemma}
If $g$ from (\ref{retract}) is a fibration then so is $f$.
\end{lemma}

\begin{proof}
For $i \in I$ we have the following commutative diagram in $\mathcal{C}$:
$$\xymatrix{
A_i \ar[d]_{f^{i}} \ar[r]^{\iota_j}&B_j\ar[d]_{g^{j}}\ar[r]^{\rho_i}&A_i\ar[d]_{f^{i}}\\
A'^{i}\ar[r]^{{\iota}^{\sim j}}&B'^{j}\ar[r]^{{\rho}^{\sim i}}&A'^{i}
}$$
where ${\iota}^{\sim j}$, ${\rho}^{\sim i}$ are maps induced by components of $\iota',\rho'$ respectively, so $\rho^{\sim i}\iota^{\sim j}=1_{A'^i}$.  Thus $f^i$ is a retraction of $g^j$ so if $g^j$ is a fibration then so is $f^i$.
\end{proof}

Consider a morphism in Prod$(\mathcal{C})$: $$f\colon \prod_{i\in I} A_i \to \prod_{j\in J} B_j$$

\begin{lemma}
The morphism $f$ may be factorized as a cofibration followed by an acyclic fibration as well as an acyclic cofibration followed by a fibration.
\end{lemma}

\begin{proof}
For each $j \in J$ we have a map $f_j\colon A_{i_j} \to B_j$.  We may factorize $f_j=\rho_j\iota_j$ where $\iota_j\colon A_{i_j}\to C_j$ is a cofibration and $\rho\colon C_j \to B_j$ is an acyclic fibration  for some object $C_j$.  Thus we may factorize $f$ as a cofibration followed by an acyclic fibration:

$$ \prod_{i\in I} A_i  \stackrel{\iota}{\to} \prod_{j\in J} C_j\stackrel{\rho}\to \prod_{j\in J} B_j.$$

Now to obtain the second factorization note that for each $i\in I$ we may factorise $f^i=\rho^i\iota_i$ where $\iota_i\colon A_i \to C_i$ is an acyclic cofibration and $\rho^i\colon C_i \to B^i$ is a fibration.  Thus we may factorise $f$ as an acyclic cofibration followed by a fibration:
$$ \prod_{i\in I} A_i  \stackrel{\iota}{\to} \prod_{i\in I} C_i\stackrel{\rho}\to \prod_{j\in J} B_j$$
\end{proof}

Consider the following commutative diagram in Prod$(\mathcal{C})$:
$$\xymatrix{\prod_{i\in I} A_i\ar[d]_f\ar[r]^s  &\prod_{k\in K} C_k\ar[d]^g \\
\prod_{j\in J} B_j \ar[r]^t& \prod_{k\in K} D_k
}$$
where $f$ is a cofibration and $g$ is an acyclic fibration.

\begin{lemma}
There exists a morphism: $$h\colon \prod_{j\in J} B_j \to \prod_{k\in K} C_k$$ making the diagram commute.
\end{lemma}

\begin{proof}
We construct $h$ by setting for each $k \in K$ the component $h_k$ to be a map which makes the following diagram commute:
$$\xymatrix{ A_i\ar[d]_{f_j}\ar[r]^{s_k}  & C_k\ar[d]^{g_k} \\
 B_j \ar[r]^{t_k}\ar@{.>}[ur]^{h_k}&  D_k
}$$
We know that $h_k$ exists as $f_j$ is a cofibration and $g_k$ is an acyclic fibration.
\end{proof}

Consider the following commutative diagram in Prod$(\mathcal{C})$:
$$\xymatrix{\prod_{i\in I} A_i\ar[d]_f\ar[r]^s  &\prod_{j\in J} C_j\ar[d]^g \\
\prod_{i\in I} B_i \ar[r]^t& \prod_{k\in K} D_k
}$$
where $f$ is an acyclic cofibration and $g$ is a fibration.

\begin{lemma}
There exists a morphism: $$h\colon \prod_{i\in I} B_i \to \prod_{j\in J} C_j$$ making the diagram commute.
\end{lemma}

\begin{proof}
We construct $h$ by setting for each $j \in J$ the component $h_j$ to be a map which makes the following diagram commute:
$$\xymatrix{ A_i\ar[d]_{f_i}\ar[r]^{s_j}  & C_j\ar[d]^{g^j} \\
 B_i \ar[r]^{t^{\sim j}}\ar@{.>}[ur]^{h_j}&  D^j
}$$
where $t^{\sim j}$ is the map induced by components of $t$.  We know that $h_j$ exists as $f_i$ is an acyclic cofibration and $g^j$ is a fibration.
\end{proof}

This completes the proof of Theorem \ref{thm:prod}, taking into account the existence of all small limits in colimits in Prod$(\mathcal{C})$ (Lemmas \ref{prodlim} and \ref{prodcolim}).  The definitions of cofibration and weak equivalence were the intuitive extensions of the CMC structure on $\mathcal{C}$, and determined the less obvious notion of fibration.

Consider a morphism $f\colon \prod_{i\in I} A_i \to 0$.  As $f$ has no components, they are all vacuously cofibrations, so $f$ is a cofibration.

\begin{lemma}\label{prodfibra}
An object $\prod_{i\in I} A_i \in${\rm Prod}$(\mathcal{C})$ is fibrant if and only if each of the $A_i$ are fibrant.
\end{lemma}

\begin{proof}
Let $f$ be as above.  The product in $\mathcal{C}$ over the empty set is the the terminal object $T \in \mathcal{C}$.  Thus for each $i\in I$ we have $f^i$ is the map $f^i\colon A_i \to T$.
\end{proof}

Another way of relating the CMC structures of $\mathcal{C}$ and Prod$(\mathcal{C})$ is through the following Quillen adjunction:

\bigskip
-- Let $F\colon \mathcal{C} \to {\rm Prod}(\mathcal{C})$ be the inclusion functor sending an object to itself (that is the product over the one element set of itself).

\bigskip
-- Let $G\colon {\rm Prod}(\mathcal{C})\to \mathcal{C}$ be the functor taking a formal product of objects in $\mathcal{C}$ to their actual product in $\mathcal{C}$. Given a morphism in Prod$(\mathcal{C})$:
$$f\colon \prod_{i\in I} A_i \to \prod_{j\in J} B_j,$$
note that the product in $\mathcal{C}$ of the $B_j$ over $j \in J$ is the same as the product over of $i \in I$ of the $B^i$.  Let $G(f)$ be the product of the $f^i$ over $i \in I$.

\bigskip
We have a natural transformation $\epsilon\colon FG\to 1_{{\rm Prod}(\mathcal{C})}$, whose components are the projections onto the corresponding factors.  We also have a natural isomorphism $\eta\colon 1_{\mathcal{C}} \to GF$  which is simply the identity map on each object of $\mathcal{C}$.  Thus $F$ is left adjoint to $G$: $F \dashv G$.

\begin{lemma}\label{Quill}
The pair $F,G$ form a Quillen adjuction.
\end{lemma}

\begin{proof}
It suffices to note that $F$ preserves cofibrations and weak equivalences (as well as fibrations).
\end{proof}

\begin{theorem}
The category of counital cocommutative dg coalgebras $\mathcal{V}$ is a CMC.
\end{theorem}

\begin{proof}
The category $\mathcal{V}^{\rm op}$ may be identified with the category of pseudo--compact algebras which by Lemma \ref{coalg} may be identified with Prod$(\mathcal{E})$.  By Theorem \ref{section3mainresult} $\mathcal{E}$ is a CMC.  By Theorem \ref{thm:prod}, the category $\mathcal{V}^{\rm op}$ and thus $\mathcal{V}$, is also a CMC.
\end{proof}

Note the terminal object in Prod$(\mathcal{E})$ which we have been denoting $0$ does in fact correspond to the $0$ algebra in $\mathcal{V}^{\rm op}$.  As as every object in $\mathcal{E}$ is fibrant, Lemma \ref{prodfibra} implies that every object in $\mathcal{V}^{\rm op}$ is fibrant.

In some ways our construction of a CMC on coalgebras may appear somewhat ad-hoc, as were it not for the factorization of pseudo--compact dg algebras provided by Theorem \ref{pseudo}, we would not be able to call on this identification.

However in this case the Quillen adjunction of Lemma \ref{Quill} may be expressed independently (after the fact) of the factorization of pseudo--compact dg algebras into Hinich algebras and acyclic algebras.  As well as the aesthetic benefit of viewing this link between the CMC structures on $\mathcal{E}$ and pseudo--compact dg algebras, this opens up the possibility of future generalization to the associative case, where the factorization may not be called on.

Clearly in this case the functor $F$ is just the inclusion of the category of Hinich algebras in the category of pseudo--compact dg algebras.  We end this section by giving a `factorization independent' construction of $G$.

Let $A$ be a pseudo--compact dg algebra, which from Theorem \ref{pseudo} we know factorizes as a product of acyclic and Hinich algebras:$$A=\prod_{i \in I} A_i$$
Then $G(A)$ is simply the product in $\mathcal{E}$ of the $A_i$.  Let $M_i \lhd A_i$ be the graded maximal ideal in each case. We must now consider two cases:

\bigskip
\noindent i) The case where $1\in A$ is not exact,

\bigskip
\noindent ii) The case where $1 \in A$ is exact.

\bigskip
In case (i) we know that at least one of the $A_i$ is a Hinich algebra, hence by Lemma \ref{nomap} we know that $G(A)$ is the product in the category of Hinich algebras of the $A_i^H$.

Let Jac$(A)$ denote the Jacobson radical of $A$.  This is simply the direct product of the $M_i$.  Let $d^{-1}$Jac$(A)$ denote the preimage under the differential of the Jacobson radical.  Recalling that in each $A_i$ we have $d^{-1}M_i=A_i^H$, we obtain:$$ d^{-1}{\rm Jac}(A)=\prod_{i \in I} A_i^H
$$

Finally to obtain the product in the category of Hinich algebras, we must take the Jacobson radical once more and reattach the unit.  We conclude:

\begin{lemma}
If $1 \in A$ is not exact then we have{\rm:}
$$
G(A)= k \oplus {\rm Jac}(d^{-1}{\rm Jac}(A)).$$
\end{lemma}

If we introduce the notation dJac$(A)$ to denote the intersection of all graded maximal \emph{differential} ideals, then we may write this more concisely.   We have that dJac$(A_i)$ is the graded maximal ideal in $A_i^H$ for each $i$ and thus dJac$(A)$ is the product of these maximal ideals.  In case (i) we may therefore write:

$$
G(A)= k \oplus {\rm dJac}(A).$$

We now consider case (ii).   In this case all the $A_i$ are acyclic algebras.  We may apply Lemma \ref{ifmap} to deduce that $G(A)$ is the product in the category of Hinich algebras of the $A_i^0$, tensored with $\Lambda(x)$.  Recalling that in each $A_i$ we have $d^{-1}0=A_i^0$, we obtain:$$ d^{-1}0=\prod_{i \in I} A_i^0
$$

Again to obtain the product in the category of Hinich algebras, we must take the Jacobson radical  and reattach the unit.  We conclude:

\begin{proposition}
If $1 \in A$ is exact then we have{\rm:}
$$
G(A)= (k \oplus {\rm Jac}(d^{-1}0)) \otimes\Lambda(x).$$
\end{proposition}

Combining with Lemma \ref{quilladjEH} we have that the inclusion of the original category of Hinich algebras in the category of all pseudo--compact dg algebras $FE$, has a right adjoint $HG$, together with which it forms a Quillen adjunction.  We may write $HG$ explicitly:

\begin{proposition}
Let $A$ be a pseudo--compact dg algebra.  We have{\rm:}
$$
HG(A)= \left\{ \begin{array}{cl}
 k \oplus {\rm dJac}(A),& 1 \in A { \rm \,\, is\,\, not\,\, exact,}\\
k \oplus {\rm Jac}(d^{-1}0)[x],& 1 \in A { \rm \,\,is\,\, exact.}\end{array}\color{white} \right|\color{black}
$$
\end{proposition}

\section{Curved Lie algebras} \label{curvedsec}

In \cite{Hini} the CMC structure on the Hinich category is induced via a Quillen equivalence with the category of dg Lie algebras.  It is natural to ask if this equivalence extends to the extended Hinich category $\mathcal{E}$, and if so, in what way we should extend the category of Lie algebras to include `acyclic' objects.  It transpires that the correct extension is the category of curved Lie algebras.

A curved Lie algebra is a graded Lie algebra $A$ containing a distinguished degree -2 element $w_A\in A$ (referred to as the curvature of $A$), and having a degree -1 derivation $d$, such that:
$$
d(ab)=(da)b + -1^{{\rm deg}\,a} a (db),\qquad d^2a=[w_A,a].
$$
We abuse terminology slightly by referring to $d$ as the {\sl differential} of $A$.

Morphisms of curved Lie algebras must respect the grading, Lie bracket, differential and distinguished element.  The usual dg  Lie algebras are precisely the curved Lie algebras $A$, satisfying $w_A=0$.  We refer to these as uncurved and to curved Lie algebras $A$ with $w_A\neq 0$ as genuinely curved or gencurved.  Clearly there are no morphisms from an uncurved Lie algebra to a gencurved one.  Denote the category of curved Lie algebras $\mathcal{G}$ and the subcategory of usual dg Lie algebras $\mathcal{L}$.

The category $\mathcal{L}$ is a CMC \cite{Quil}.  In particular its CMC structure has fibrations precisely the surjective maps \cite[Theorem 5.1]{Quil}.

\begin{remark}
We note at this point that the arguments for the extension of a CMC structure on dg Lie algebras to curved Lie algebras may be applied verbatim to extending the CMC structure on dg associative algebras to curved associative Lie algebras.

This raises the distinct possibility that a CMC structure on conilpotent coassociative coalgebras could be transferred from the category of curved associative algebras, by extending our adjoint pair of functors between $\mathcal{E}$ and $\mathcal{G}$.
\end{remark}

In order for it to make sense to say that $\mathcal{E}$ is Quillen equivalent to $\mathcal{G}$ we must first establish that $\mathcal{G}$ is a CMC.

\begin{lemma}
The category $\mathcal{G}$ is closed under taking small limits.
\end{lemma}

\begin{proof}
Given a small diagram $\mathcal{D}$, one may regard it as a diagram of graded vector spaces.  Let $L$ denote the limit of this diagram.  There is a natural differential and Lie bracket induced on $L$.  Finally let the curvature $w_L\in L$ be the unique element which maps to the curvature of each object in $\mathcal{D}$.  Then any cone from a curved Lie algebra $X$ to $\mathcal{D}$ will factor through a unique map of curved Lie algebras $X \to L$.
\end{proof}

Note that the product of curved Lie algebras has the direct sum of the underlying vector spaces as its underlying vector space.  We will use the notation $\times$ to denote the product of curved Lie algebras.  Also we will denote the curved Lie algebra freely generated by certain generators  by placing the generators in angled brackets $\langle \rangle$.  Conversely if elements of a curved Lie algebra already under consideration are placed in angled brackets $\langle \rangle$, then  it will denote the ideal generated by those elements.

\begin{lemma}
The category $\mathcal{G}$ is closed under taking small colimits.
\end{lemma}

\begin{proof}
Given a small diagram $\mathcal{D}$,  let $D_w$ denote the extension of $\mathcal D$ which includes the initial object $\langle w \rangle$ and its morphism to each object of $\mathcal{D}$.  Let $C$ denote the colimit of the underlying diagram of graded Lie algebras.  Define a differential on $C$ inductively via the Leibniz rule.  Further let $w_C$ the curvature of $C$ be the image of $w \in \langle w \rangle$.

It remains to verify that $d^2x=[w_C,x]$ for all $x \in C$.  This follows by induction on the largest length of a term in $x$:  Let $a,b \in C$ satisfy $d^2a=[w_C,a],\,d^2b=[w_C,b]$.  From the Jacobi identity we have:
$$
d^2[a,b]=[d^2a,b]+[a,d^2b]=[[w_C,a],b] + [a,[w_C,b]]=[w_C,[a,b]].
$$
\end{proof}

We next define a CMC structure on $\mathcal{G}$, extending the one from \cite{Quil} on $\mathcal{L}$ {\rm :}

\begin{definition} \label{cmccurved}
We define weak equivalences, fibrations and cofibrations in $\mathcal{G}$ as follows:

\bigskip
\noindent -- Weak equivalences are the weak equivalences in $\mathcal{L}$, together with all maps between gencurved Lie algebras.

\bigskip
\noindent -- Fibrations are all surjective maps.

\bigskip
\noindent -- Cofibrations are all maps which have the left lifting property with respect to acyclic fibrations (as defined above).

\end{definition}

From the definition it is clear that all three classes contain all identity maps and are closed under composition and retraction.  Further it is clear that weak equivalences satisfy the 2 of 3 rule and that cofibrations satisfy the left lifting property with respect to acyclic fibrations.  Finally note that all three classes restricted to $\mathcal{L}$ are precisely the corresponding class in the CMC sturcture on $\mathcal{L}$ from \cite[Theorem 5.1]{Quil}.

\begin{lemma}\label{acyccofllp}
Acyclic cofibrations are precisely the maps which have the left lifting property for all fibrations.
\end{lemma}

\begin{proof}
First we will show that acyclic cofibrations have the left lifting property for all fibrations.  Let $f\colon A \to B$ be an acyclic cofibration.  If $A,B$ are uncurved then $f$ has the required lifting property for all fibrations between uncurved Lie algebras and vacuously for all other fibrations.

We may then assume that $A,B$ are gencurved.  Suppose we have a commutative square in $\mathcal{G}$:
$$\xymatrix{ A\ar[d]_{f}\ar[r]^u & X\ar[d]^{g} \\
 B \ar[r]_v&  Y
}$$
where $g$ is surjective.  We then have the following commutative square:
$$\xymatrix{ A\ar[d]_{f}\ar[r]^{(u,f)} & X\times B \ar[d]^{g \times 1_B} \\
 B \ar[r]_{(v,1_B)}&  Y \times B
}$$
The right hand map is still surjective, but is now acyclic as it is a map between gencurved Lie algebras.  As $f$ is a cofibration, we then have a lifting $(h,1_B)\colon B \to X \times B$ making the diagram commute.  In particular the following diagram commutes:

$$\xymatrix{ A\ar[d]_{f}\ar[r]^u & X\ar[d]^{g} \\
 B \ar@{.>}[ur]^{h} \ar[r]_v&  Y
}$$

Hence $f$ satisfies the required lifting property.  Conversely suppose that $f\colon A \to B$ satisfies the left lifting property with respect to all fibrations.  In particular it has the left lifting property for all acyclic fibrations, hence $f$ is a cofibration.  It remains to show that $f$ is acyclic.

If $A,B$ are both gencurved then clearly $f$ is acyclic.  If $A,B$ are both uncurved then as $f$ satisfies the left lifting property for all fibrations in $\mathcal{L}$, we know that $f$ is acyclic.  It will then suffice to rule out the possibility that $A$ is gencurved and $B$ is uncurved.

Suppose that $A$ is gencurved and consider the following commutative square in $\mathcal{G}$:
$$\xymatrix{ A\ar[d]_{f}\ar[r]^{(1_A,f)} & A\times B \ar[d]\\
 B \ar[r]_{1_B}&  B
}$$
where the right hand morphism is projection onto $B$ and hence a fibration.  We have a lifting $(h,1_B)\colon B \to A \times B$ making the diagram commute.  In particular $hf=1_A$ so $h(w_B)=w_a \neq 0$.  Thus $w_B \neq 0$ and $B$ is not uncurved.
\end{proof}

\begin{corollary}\label{pushcof}
Acyclic cofibrations are closed under pushouts.
\end{corollary}

\begin{proof}
Any class of morphisms which is defined by having the left lifting property with respect to a class of morphisms will be closed under pushouts.
\end{proof}

\begin{corollary}
Given a cofibration $f\colon A \to B$ between gencurved Lie algebras, we have that the induced map $f_w\colon A/\langle w_A \rangle \to B/\langle w_B \rangle$ is an acyclic cofibration in $\mathcal{L}$.
\end{corollary}

\begin{proof}
Suppose we have a commutative square:
$$\xymatrix{ A/\langle w_A\rangle\ar[d]_{f_w}\ar[r] & X \ar[d]^g\\
B/\langle w_B\rangle \ar[r]&  Y
}$$
where $g$ is a fibration in $\mathcal{L}$.

We may extend the diagram in the following way so that $f$ having the left lifting property with respect to $g$ yields a lifting which necessarily factors through a map $h\colon B/\langle w_B \rangle\to X$:
$$\xymatrix{ A\ar[d]_f\ar[r]&A/\langle w_A\rangle\ar[d]_{f_w}\ar[r] & X \ar[d]^g\\
B\ar[r] &B/\langle w_B\rangle\ar@{.>}[ur]^{h} \ar[r]&  Y
}$$
As the quotient maps $A \to A/\langle w_A \rangle$, $B \to B/\langle w_B \rangle$ are surjective we have that the above diagram commutes.  Thus $f_w$ has the left lifting property for all fibrations in $\mathcal{L}$, as required.
\end{proof}

In order to show that we have a CMC structure on $\mathcal{G}$ it remains to show that we have the required factorizations.

Let $B$ be a curved Lie algebra and let $B_\_$ denote its homogeneous elements.  We define $$F_B=\langle w, u_b,v_b \vert \, b \in B_\_ \rangle,$$ to be the free curved Lie algebra on $B$.  In particular $w$ is the curvature of $F_B$ and the only relations imposed on the generators are that for all $b\in B$ we have $du_b=v_b$ and $dv_b=[w,u_b]$.  The elements of $F_B$ are graded so that we have the natural map $m_B \colon F_B \to B$ sending: \begin{eqnarray*} u_b &\mapsto& b,\\v_b &\mapsto& db,\\w &\mapsto& w_B\end{eqnarray*} for each $b \in B_\_$.

\begin{lemma} \label{wtoFB}
The inclusion $\langle w \rangle\to F_B$ is an acyclic cofibration.
\end{lemma}

\begin{proof}
Given a commutative square:
$$\xymatrix{\langle w \rangle \ar[d]\ar[r]^{j} & X \ar[d]^g\\
 F_B \ar[r]_k& Y
}$$
with $g$ a surjective map, we may define a lift $h\colon F_B \to X$ making the diagram commute as follows:

\bigskip
--Map $w$ to $w_X$.

--Map each $u_b$ to a preimage $x_b$ (of the appropiate degree) under $g$ of $k(u_b)$.

--Map each $v_b$ to $dx_b$.
\end{proof}

\begin{lemma}\label{killw}
For any curved Lie algebra $A$, the quotient map $A\to A/\langle w_A\rangle$ is a cofibration.
\end{lemma}

\begin{proof}
The map $A\to A/\langle w_A\rangle$ has the left lifting property for all maps between uncurved Lie algebras and (vacuously) for all maps between gencurved Lie algebras.  In particular it has the left lifting property for all acyclic fibrations as required.
\end{proof}

\begin{lemma} \label{factorliealgebramaps}
Let $f\colon A \to B$ be a map of curved Lie algebras.  Then $f$ may be factorised as{\rm :}

\noindent i) an acyclic cofibration followed by a fibration,

\noindent ii) a cofibration followed by an acyclic fibration.
\end{lemma}

\begin{proof}
The maps $f$ and $m_B$ induce a map $j\colon F_B {\star_{\tiny{\langle w \rangle}} A} \to B$ on the pushout $F_B {\star_{\langle w \rangle}} A$ making the following diagram commute:
$$\xymatrix{\langle w \rangle \ar[d]\ar[r] & A\ar@/_-1pc/[ddr]^f \ar[d]^i\\
 F_B\ar@/_1pc/[drr]_{m_B} \ar[r]& F_B {\star_{\langle w \rangle}} A \ar@{.>}[dr]^j \\
& & B
}$$
By Lemma \ref{wtoFB} the left hand vertical map is an acyclic cofibration and by Corollary \ref{pushcof}, its pushout $i$ is also an acyclic cofibraion.  The induced map $j$ is clearly surjective, hence a fibration.  Thus $f=ji$ gives us the first factorization (i).

If $A,B$ are gencurved, then the first factorization is also the second factorization (ii), as all maps between gencurved Lie algebras are weak equivalences.  If $A,B$ are uncurved then the second factorization follows from $\mathcal{L}$ being a CMC.  Finally we consider the case where $A$ is gencurved and $B$ is uncurved.

In this case the map $f$ factorizes through the quotient map $A \to A/\langle w \rangle$, which by Lemma \ref{killw} is a cofibraion.  The map $A/\langle w \rangle \to B$ lies in $\mathcal{L}$ and  hence factorizes through some uncurved Lie algebra $C$, as a cofibraion $i$ followed by an acyclic fibration $j$.  Thus we obtain the second factorization (ii) of $f$:
$$\xymatrix{A \ar[r]& A/\langle w \rangle\ar[r]^i&C\ar[r]^j& B}$$
\end{proof}

Thus $\mathcal{G}$ contains limits and colimits and has a closed model structure.  We conclude:

\begin{theorem}
The category $\mathcal{G}$  is a CMC.
\end{theorem}

We can give a `generating set' (in the sense of Lemma \ref{gencofG} below) for the cofibrations of $\mathcal{G}$.  Let $I$ be the set of morphisms in $\mathcal{G}$ of the following four types:

\bigskip
--(i) Inclusions $\langle w \rangle \to \langle w,u,v \rangle$ where $du=v$, $w$ is the curvature element and $u$ may be of any degree.

\bigskip
--(ii) Inclusions $\langle v \rangle \to \langle u,v \rangle$ where $dv=0$, $du=v$, 0 is the curvature element and $u$ may be of any degree.

\bigskip
--(iii) Inclusions $0 \to \langle u \rangle$, where $du=0$, $0$ is the curvature element and $u$ may be of any degree.

\bigskip
--(iv) The map $\langle w \rangle \to 0$, where $w$ is the curvature element.

\bigskip
From Lemma \ref{wtoFB} we know that type (i) morphisms in $I$ are cofibrations.  From Lemma \ref{killw} we know that the type (iv) morphism in $I$ is a cofibration.  Also we are given that type (ii) and (iii) morphisms in $I$ are cofibrations in $\mathcal{L}$ \cite{Quil}.  Thus the morphisms in $I$ are cofibrations.  In fact they generate all cofibrations in $\mathcal{G}$ in the following sense:

\begin{lemma}\label{gencofG}
The maps satisfying the right lifting property with respect to $I$ are precisely the acyclic fibrations of $\mathcal{G}$.
\end{lemma}

\begin{proof}
As the morphisms of $I$ are cofibrations it is clear that acyclic fibrations will have the right lifting property with respect to them.  We must prove the converse.

Let $f\colon A \to B$ have the right lifting property with respect to $I$.  If $A,B$ are uncurved then for $f$ to have the right lifting property with respect to type (ii) and type (iii) morphisms in $I$, it must be an acyclic fibration in $\mathcal{L}$.

In order to have the right lifting property for the type (iv) morphism in $I$, $f$ cannot be a map from a gencurved algebra to an uncurved algebra.

Finally, if both $A,B$ are gencurved then $f$ is acyclic.  Further, having the right lifting property with respect to type (i) morphsms in $I$ implies that $f$ is surjective, hence a fibration as well as acyclic.
\end{proof}

A description of the cofibrations in $\mathcal{G}$ in terms of how they extend the cofibrations in $\mathcal{L}$ is given by the following:

\begin{lemma}
Every cofibration in $\mathcal{G}$ is a retract of a pushout of a cofibration of one of the following three types:

\bigskip
1) A cofibration in $\mathcal{L}$,

2) The composition of a quotient map $A\!\to\! A/\langle w_A\rangle$ with a cofibration in $\mathcal{L}$,

3) An inclusion of the form $\langle w \rangle\to F_X$ for some graded set $X$.

\end{lemma}

\begin{proof}
In the proof of Lemma \ref{factorliealgebramaps} we show that every morphism in $\mathcal{G}$ factors as a cofibration of type 1) or 2) or a pushout of a cofibration of type 3), followed by an acyclic fibration.

In particular for any cofibration $f$, we have $f=\rho i$ where $\rho$ is an acyclic fibration and $i$ is a pushout of a cofibration of one of the three types.

Then by the standard diagram chase we have that $f$ is a retract of $i$.
\end{proof}

Having shown that $\mathcal{G}$ is a CMC,  we will now extend the contravariant functors in the Quillen equivalence between Hinich algebras and Lie algebras \cite{Hini}, to a Quillen equivalence: \begin{eqnarray}\label{curveq}
\mathcal{E}\stackrel {CE} \leftrightarrows_{L} \mathcal{G}. \label{junction}
\end{eqnarray}

Here $CE$ takes a curved Lie algebra, $\mathfrak{g}$ to its Chevalley--Eilenberg complex.  As an algebra this is the completed free graded commutative algebra on the dual of $\mathfrak{g}$ raised one degree: $\hat{S}(\Sigma\mathfrak{g})^*$.

The differential $d$ is induced via the Leibniz rule and continuity by its restriction to $(\Sigma \mathfrak{g})^*$.  This restriction is given by:$$d=d_0+d_1+d_2,$$
where:

\bigskip
$d_0\colon (\Sigma\mathfrak{g})^* \to k$ is given by evaluation on $w$ the curvature of $\mathfrak{g}$,

$d_1\colon(\Sigma\mathfrak{g})^* \to(\Sigma\mathfrak{g})^*$ is given by precomposition with the differential on $\mathfrak{g}$,

$d_2\colon(\Sigma\mathfrak{g})^* \to S^2(\Sigma\mathfrak{g})^*$ is given by precomposition with the Lie  bracket on $\mathfrak{g}$.

\bigskip
Note that on $(\Sigma\mathfrak{g})^*$ we have $d_0d_1=0$ as the curvature is closed, $d_1^2+d_0d_2=0$ by the defining property of the curvature and $d_1d_2+d_2d_1=0$ by the Leibniz rule.  It follows that $d^2=0$, as the remaining terms vanish.

The functor $L$ takes an algebra $A \in \mathcal{E}$ to the free Lie algebra on the topological dual of $A$ raised one degree: ${L}(\Sigma A)^*$.  Composition of the differential on $A$ with augmentation gives an element in $(\Sigma A)^*$, which we set as the curvature of $LA$.

\begin{theorem}\label{equivcurved}
The functors $L$ and $CE$ in (\ref{curveq}) determine a contravariant Quillen equivalence between the closed model categories of curved Lie algebras and the extended Hinich category.
\end{theorem}
The proof consists of a succession of lemmas below.
\begin{lemma}\label{acyccurved}
The functors $L$ and $CE$ interchange Hinich algebras and uncurved Lie algebras.  They also interchange acyclic algebras and gencurved Lie algebras.
\end{lemma}

\begin{proof}
If $A$ is a Hinich algebra, then the image of its differential lies in its maximal ideal so the composition of the differential with augmentation is zero. Thus $LA$ is uncurved.   Conversely if the maximal ideal of $A$ is closed under the differential then $A$ is a Hinich algebra.

If $\mathfrak{g}$ is a gencurved Lie algebra then $d_0$ is non-zero and for some $x\in (\Sigma \mathfrak{g})^*$ we have $dx$ not in the maximal ideal of $CE(\mathfrak{g})$.  It follows from the proof of Lemma \ref{twotypes} that $CE(\mathfrak{g})$ is acyclic. Conversely if $CE(\mathfrak{g})$ is acyclic, then $d_0$ is non-zero and $\mathfrak{g}$ is gencurved.
\end{proof}

Let $\epsilon\colon CEL\to 1_{\mathcal{E}}$ be the unit of the adjunction and let $\delta\colon LCE \to 1_{\mathcal{G}}$  be the counit.  In order to show that $CE, L$ form a Quillen equivalence, it suffices to show that $\epsilon, \delta$ are always weak equivalences and that $CE$ takes fibrations to cofibrations and acyclic fibrations to acyclic cofibrations.

\begin{lemma}
For any $A\in \mathcal{E}$ and any $\mathfrak{g}\in \mathcal{G}$ we have that $\epsilon_A$ and $\delta_{\mathfrak{g}}$ are weak equivalences.
\end{lemma}

\begin{proof}
As the equivalence has been established between Hinich algebras and uncurved Lie algebras \cite{Hini}, we may assume that $A$ is acyclic and $\mathfrak{g}$ is gencurved.  Then by Lemma \ref{acyccurved} we know that $\epsilon_A$ is a map between acyclic algebras and $\delta_{\mathfrak{g}}$ is a map between gencurved Lie algebras.  Thus they are both weak equivalences (see Definitions \ref{cmc}, \ref{cmccurved}).
\end{proof}

We know that $CE$ takes fibrations between uncurved Lie algebras to cofibrations and preserves their acyclicity \cite{Hini}.  Further we know that $CE$ preserves the acyclicity of all maps outside the subcategory of uncurved Lie algebras.  That is we know that the only weak equivalences outside this subcategory are maps between gencurved algebras, and $CE$ takes these to maps between acyclic algebras, which are known to all be weak equivalences.

It remains to show that:

\bigskip
(i) $CE$ takes fibrations between gencurved Lie algebras to cofibrations.

(ii) $CE$ takes fibrations that are from gencurved to uncurved Lie algebras to cofibrations.

\bigskip
Recall that fibrations in $\mathcal{G}$ are precisely the surjective maps.  Let $f\colon \mathfrak{g} \to \mathfrak{h}$ be a surjective map in $\mathcal{G}$.

\begin{lemma}
If $\mathfrak{g},\mathfrak{h}$ are gencurved then $CE(f)$ is a cofibration.
\end{lemma}

\begin{proof}
From Lemma \ref{acyccurved} we know that $CE(\mathfrak{g})$ and $CE(\mathfrak{h})$ are acyclic.  Thus by Lemma \ref{acyc} we may write $CE(f)=f'\otimes 1_{\Lambda(x)}$ for some $f'\colon A \to B$ a map of  Hinich algebras with vanishing differentials.  Our goal will be to show that $f'$ is a cofibration in the Hinich category, making $CE(f)$ a cofibration in $\mathcal{C}$ of type (ii) (see Definition \ref{cmc}).

As $f$ is surjective, we have that $f^*\colon (\Sigma \mathfrak{h})^* \to (\Sigma \mathfrak{g})^*$ is injective and thus splits as a map of graded vector spaces. This splitting induces a map of graded algebras (not necessarily respecting the differentials): $$j\colon CE(\mathfrak{g}) \to CE(\mathfrak{h}),$$ satisfying $jCE(f)=1_{CE(\mathfrak{h})}$.  Let $j'\colon B \to A$ be the restriction of $j$ to $B$,  composed with the natural map of algebras (again not necessarily respecting the differential) $CE(h) \to A$.  As both the differentials on $A$ and $B$ vanish, we have that $j'$ respects the differentials and is a map of Hinich algebras.  Clearly we have $j'f'=1_A$, so we know that $f'$ is a retraction in the Hinich category.

In order to conclude that the retraction $f'$ is a cofibration in the Hinich category it suffices to show that as a graded algebra, $B$ has the form $B=\hat{S}(U)$ for some vector space $U$.  This follows from the fact that as a graded algebra (ignoring the differentials) $B$ is a retract of $CE(\mathfrak{g})$.  Indeed the inclusion of $B$ in $CE(\mathfrak{g})$ composed with the natural map of algebras $CE(\mathfrak{g}) \to B$ killing $x$, results in the identity map on $B$.
\end{proof}

\begin{lemma}
If $\mathfrak{g}$ is gencurved and $\mathfrak{h}$ is uncurved then $CE(f)$ is a cofibration.
\end{lemma}

\begin{proof}
We have an exact sequence of Lie algebras:$$I \stackrel i \to \mathfrak{g} \stackrel  f \to  \mathfrak{h},$$
where $i$ is the inclusion of the kernel of $f$.  As $I$ contains the curvature of $\mathfrak{g}$ we have that $i$ is a map of gencurved Lie algebras.

Applying $CE$ to this sequence we get:
\begin{eqnarray}CE(I) \stackrel {CE(i)} \longleftarrow CE(\mathfrak{g}) \stackrel  {CE(f)} \longleftarrow  CE(\mathfrak{h}).\label{CEseq}\end{eqnarray}

The map $i$ induces a surjective map $(\Sigma \mathfrak{g})^*\to (\Sigma I)^*$ which splits as a map of graded vector spaces.  Thus we have a retraction of graded algebras induced $j\colon  \hat{S} (\Sigma I)^*  \to  \hat{S}(\Sigma \mathfrak{g})^*$, satisfying $CE(i)j=1_{CE(i)}$.

Lemma \ref{acyccurved} implies that $CE(I), CE(\mathfrak{g})$ are acyclic, so by Lemma \ref{acyc} we may we may write $$CE(i)=i'\otimes 1_{\Lambda(x)}\colon B \otimes \Lambda(x) \leftarrow A \otimes \Lambda(x),$$ for a map of Hinich algebras $i'\colon A \to B$.

Let $j'\colon B \to A$ be the restriction of $j$ to $B$, composed with the projection $A\otimes \Lambda(x) \to A$ killing $x$.  Then we have that $i'j'=1_B$ and $$(i'\otimes 1_{\Lambda(x)})(j'\otimes 1_{\Lambda(x)})=1_{CE(i)}.$$  Now $j'$ is a map of Hinich algebras, as it respects the (trivial) differentials of $A,B$.   Consequently $(j'\otimes 1_{\Lambda(x)})$ is a morphism in $\mathcal{E}$ and the sequence (\ref{CEseq}) splits:

\begin{center}
\xymatrix{
\qquad\qquad\qquad\qquad{\hat{S}(\Sigma I)^*} \ar@/_1pc/[r]_{j'\otimes 1_{\Lambda(x)}}\quad&\quad   \hat{S}(\Sigma I)^* \otimes \hat{S} (\Sigma \mathfrak{h})^* \ar[l]_{\qquad\,\,\,CE(i)}&   \hat{S}(\Sigma \mathfrak{h})^*\ar[l]_{\qquad CE(f)}.}
\end{center}

We see that $CE(f)$ is the coproduct of the identity map $1_{\hat{S}(\Sigma \mathfrak{h})^*}$ with the inclusion $\iota\colon k \to \hat{S}(\Sigma I)^*$.  It remains to show that $\iota$ is a cofibration.

As the differential on $k$ vanishes, we have that $\iota$ factors through $A$.  That is $\iota$ is the tensor product of some map $z\colon k \to A$ with the natural inclusion $k\to \Lambda(x)$.
To conclude that $\iota$ is a cofibration in $\mathcal{E}$ of type (iii) we need only show that $z$ is a cofibration in the Hinich category (see Definition \ref{cmc}).

Note that $z$ is a retract as the Hinich algebra $A$ is augmented over $k$.  We have that $A$ is a retract (as a graded algebra) of  $\hat{S} (\Sigma I)^*$, so the retract $z\colon k \to A$ is indeed a cofibration in the Hinich category.
\end{proof}

This completes the proof of Theorem \ref{equivcurved}.
Combining this with Lemma \ref{coalg} we obtain:
\begin{corollary}\label{Koszul}
The category of $\mathcal V$ of counital cocommutative dg coalgebras is Quillen equivalent to the category $\operatorname{coProd}\mathcal G$ of formal coproducts of curved Lie algebras.
\end{corollary}

\end{document}